\documentclass[11pt]{amsart}
\usepackage{etex}
\usepackage{amssymb}
\usepackage{amscd}
\usepackage{amsthm}
\usepackage[backref=page]{hyperref}
\usepackage{extpfeil}
\usepackage{pictexwd,dcpic}
\usepackage[margin=1in]{geometry}
\usepackage{tikz}
\usetikzlibrary{matrix,arrows,decorations.pathmorphing}
\usepackage[all]{xy}

\usepackage{mathtools}

\numberwithin{equation}{section}
\theoremstyle{plain}
\newtheorem{thm}[equation]{Theorem}
\newtheorem{prop}[equation]{Proposition}
\newtheorem{lem}[equation]{Lemma}

\theoremstyle{definition}
\newtheorem{defn}[equation]{Definition}
\newtheorem{exm}[equation]{Example}
\theoremstyle{remark}
\newtheorem{rem}[equation]{Remark}
\renewcommand{\le}{\leqslant}
\renewcommand{\ge}{\geqslant}

\newcommand\blfootnote[1]{
  \begingroup
  \renewcommand\thefootnote{}\footnote{#1}
  \addtocounter{footnote}{-1}
  \endgroup}

\newcommand{\m}{\mathfrak{m}}
\newcommand{\Z}{\mathbb{Z}}
\newcommand{\Q}{\mathbb{Q}}
\newcommand{\R}{\mathbb{R}}
\newcommand{\C}{\mathbb{C}}

\newcommand \id {{\operatorname{id}}}
\newcommand\MF{{\operatorname{MF}}}
\newcommand\Hom{{\operatorname{Hom}}}
\newcommand\cone{{\operatorname{cone}}}

\newcommand\Ob{{\operatorname{Ob}}}
\newcommand\ST{{\operatorname{ST}}}
\newcommand\IHS{{\operatorname{IHS}}}

\newcommand\Spec{{\operatorname{Spec}}}

\newcommand\Perf{{\operatorname{Perf}}}

\newcommand\fgmod{{\operatorname{mod}}}

\newcommand\End{{\operatorname{End}}}

\newcommand{\fL}{\mathcal{L}}
\newcommand{\fC}{\mathcal{C}}
\newcommand{\fE}{\mathcal{E}}

\newcommand{\wQ}{\widehat{Q}}
\newcommand{\into}{\hookrightarrow}
\newcommand{\cliff}{\operatorname{Cliff}}
\newcommand{\wST}{\widehat{\text{ST}}}

\newcommand{\largewedge}{\mbox{\Large $\wedge$}}
\newcommand{\Mat}{\operatorname{Mat}}
\newcommand{\chara}{\operatorname{char}}
\newcommand{\ABS}{\operatorname{ABS}}
\newcommand{\BEH}{\operatorname{BEH}}

\newcommand{\stab}{\operatorname{stab}}

\newcommand{\coker}{\operatorname{coker}}

\newcommand{\MCM}{\operatorname{MCM}}
\newcommand{\DD}{\operatorname{D}^{\operatorname{b}}}

\newcommand{\Tor}{\operatorname{Tor}}

\newcommand{\Iso}{\operatorname{Iso}}
\newcommand{\tors}{\operatorname{tors}}

\def\darrow#1#2{\xtofrom[#2]{#1}}

\begin{document}
\title{Kn\"orrer Periodicity and Bott Periodicity}
\author{Michael K. Brown}
\begin{abstract}
The goal of this article is to explain a precise sense in which
Kn\"orrer periodicity in commutative algebra and Bott periodicity in
topological $K$-theory are compatible phenomena. Along the way, we prove an 8-periodic version of Kn\"orrer
periodicity for real isolated hypersurface singularities, and we
construct a homomorphism from the Grothendieck group of the homotopy category of matrix factorizations of a complex (real)
polynomial $f$ into the topological $K$-theory of its Milnor fiber
(positive or negative Milnor fiber). 
\end{abstract}
\maketitle

\tableofcontents

\blfootnote{\emph{Date:} \today}
\blfootnote{The author was partially supported by NSF Award DMS-0966600 and the University of Nebraska-Lincoln MCTP
grant (NSF Award DMS-0838463).}

\section{Introduction}
Let $k$ be a field. In this article, we study hypersurface rings of the form
$k[x_1, \dots, x_n]/(f)$ from
both an algebraic and topological point of view. An important algebraic invariant of such a ring is its homotopy category of
matrix factorizations, which we denote by $[\MF(k[x_1, \dots, x_n],
f)]$ (we recall the
definition of this
category in Section~\ref{mf}). Matrix factorizations were introduced by Eisenbud in
\cite{eisenbud1980homological} as a tool for studying the homological
behavior of modules over a hypersurface ring. More recently,
matrix factorizations have begun appearing in a wide variety of
contexts, for instance homological mirror symmetry
(e.g. \cite{katzarkov2008hodge}, by Katzarkov-Kontsevich-Pantev) and
knot theory (e.g. \cite{khovanov2008matrix}, by Khovanov-Rozansky). In
the present work, we continue the study of an interplay
between matrix factorizations and topological $K$-theory that was begun
in the inspiring paper \cite{buchweitz2012index} of Buchweitz-van Straten. 

A fundamental result in the theory of matrix factorizations is
Kn\"orrer's periodicity theorem:
\begin{thm}[\cite{knorrer1987cohen} Theorem 3.1]
\label{kpint}
Suppose $k$ is algebraically closed and $\chara(k) \ne
2$. If $f \in (x_1, \dots, x_n)
\subseteq k[[x_1, \dots, x_n]]$,
there is an equivalence of categories
$$[\MF(k[[x_1, \dots, x_n]],f)] \xrightarrow{\cong}
[\MF(k[[x_1, \dots, x_n,u,v]],f + u^2 + v^2)].$$
\end{thm}

This result plays an important role in the classification of local
hypersurface rings of finite maximal Cohen-Macaulay type; we refer the
reader to Chapter 9 of Leuschke-Wiegand's text \cite{leuschke2012cohen} for
details. Kn\"orrer's periodicity theorem also demonstrates that one cannot recover $f$ from its homotopy category
of matrix factorizations. 

The main goal of this article is explain a precise sense in which
Kn\"orrer periodicity is a manifestation of Bott periodicity in
topological $K$-theory. In Section~\ref{ch1}, we motivate this project
with a proof of an 8-periodic version of Kn\"orrer
periodicity for isolated hypersurface singularities over the real
numbers:

\begin{thm}
\label{thm1}
Let $f \in (x_1, \dots, x_n) \subseteq \R[x_1, \dots, x_n]$, and suppose $\R[x_1, \dots, x_n]/(f)$ has
an isolated singularity at the origin (i.e. $\dim_\R \frac{\R[[x_1, \dots,
x_n]]}{(\frac{\partial f}{\partial x_1}, \dots, \frac{\partial
  f}{\partial x_n})} < \infty$). Then there exists
an equivalence of triangulated categories
$$[\MF(\R[[x_1, \dots, x_n]],f)] \xrightarrow{\cong}
[\MF(\R[[x_1, \dots, x_n, u_1, \dots, u_8]],f - u_1^2 - \cdots - u_8^2)].$$
\end{thm}

We point out that the  ``period'' here is exactly 8; that is, for $1
\le l < 8$, it can happen that $$[\MF(\R[[x_1, \dots, x_n]], f)] \ncong
[\MF(\R[[x_1, \dots, x_n, u_1, \dots, u_l]], f
- u_1^2 - \cdots - u_l^2)].$$

Our proof relies heavily on machinery developed by Dyckerhoff
and To\"en in \cite{dyckerhoff2011compact} and
\cite{toen2007homotopy}. This result draws a distinction between the maximal Cohen-Macaulay
representation theory of hypersurface rings with ground field $\R$ and those
whose ground field is algebraically closed and has characteristic not
equal to 2, since the latter exhibit 2-periodic Kn\"orrer
periodicity. The maximal Cohen-Macaulay
representation theory of hypersurface rings with ground field $\R$
does not seem
to be well-studied, and we hope this work motivates further
investigation in
this direction.

The presence of 2- and
8-periodic versions of Kn\"orrer periodicity over $\C$ and $\R$,
respectively, suggests the
possibility of a compatibility between Kn\"orrer periodicity and
Bott periodicity. Such a compatibility statement is formulated and
proved in Section~\ref{ch2}. We state here the version of this result
over $\C$; a version over $\R$ is also proven in
Section~\ref{ch2} (Theorems~\ref{KPBPR} and~\ref{KPBP}).

\begin{thm}
\label{thm2}
Suppose $f \in (x_1, \dots, x_n) \subseteq \C[x_1, \dots, x_n]$
and either $\C[x_1, \dots, x_n]/(f)$ has an isolated singularity at the origin (i.e. $\dim_\C\frac{\C[[x_1, \dots,
x_n]]}{(\frac{\partial f}{\partial x_1}, \dots, \frac{\partial
  f}{\partial x_n})}< \infty$) or $f$ is quasi-homogeneous. Then there exists a commutative diagram
 $$
\xymatrix{
K_0[\MF(\C[x_1, \dots, x_n], f)] \ar[r]^-{\phi^\C_f} \ar[dd]^-{K}
& KU^0(B_\epsilon, F_f) \ar[d]^-{\beta} \\
& KU^0(B_\epsilon, F_f) \otimes KU^0(B_{\epsilon'}, F_{u^2+v^2}) \ar[d]^-{\ST_{KU}}\\
K_0[\MF(\C[x_1, \dots, x_n,u,v],
f + u^2+v^2)] \ar[r]^-{\phi^\C_{f + u^2 + v^2}} & KU^0(B_{\epsilon''}, F_{f
  + u^2+v^2})
}
$$
where $F_f$, $F_{u^2 + v^2}$, and $F_{f+u^2 + v^2}$ denote the Milnor
fibers of $f$, $u^2 + v^2$, and $f+u^2+v^2$; $\epsilon, \epsilon', \epsilon'' >0$;
$B_\epsilon$, $B_{\epsilon'}$, and  $B_{\epsilon''}$ are closed
balls of radius $\epsilon$, $\epsilon'$, and $\epsilon''$ in $\C^n$,
$\C^2$, and $\C^{n+2}$, respectively; $K$ is induced by the Kn\"orrer
functor; $\beta$ is the Bott
periodicity isomorphism; and $\ST_{KU}$ is given by the product in
relative $K$-theory followed by the inverse of the map induced by pullback along the Sebastiani-Thom
homotopy equivalence.
\end{thm}

The Sebastiani-Thom homotopy equivalence to which we refer in
Theorem~\ref{thm2} is discussed in Section~\ref{TStopp}.

The key construction in this section yields the
horizontal maps above; specifically, given a polynomial $f$ over the
complex (real) numbers, we build a map $\Phi^\C_f$ ($\Phi^\R_f$) that assigns
to a matrix factorization of a complex (real) polynomial $f$ a class in the topological $K$-theory of the Milnor fiber (positive or
negative Milnor fiber) of $f$; this map first appeared in
\cite{buchweitz2012index} in the setting of complex isolated
hypersurface singularities. We prove that this construction
induces a map $\phi^\C_f$ ($\phi^\R_f$) on the Grothendieck group of the homotopy
category of matrix factorizations of $f$, and we show that it
recovers the Atiyah-Bott-Shapiro construction when $f$ is a
non-degenerate quadratic over $\R$ or $\C$. The Atiyah-Bott-Shapiro
construction, introduced in Part III of \cite{atiyah1964clifford}, provides the
classical link between $\Z/2\Z$-graded modules over Clifford algebras and vector
bundles over spheres; the maps $\phi^\C_f$ and $\phi^\R_f$ we discuss in Section~\ref{ch2}
can be thought of as providing a
more general link between algebra and topology.

{\bf  Acknowledgements.} This work is adapted from my Ph.D. thesis at
the University of Nebraska-Lincoln. I must first of all thank my thesis advisor, Mark E.
Walker, for his support during my time as a graduate student at Nebraska. I thank Luchezar Avramov and Brian Harbourne for
their comments on preliminary versions of this paper, and also
Ragnar Buchweitz, Jesse Burke, Michael
Hopkins, and Claudia Miller for valuable conversations with regard to
this work. I owe special thanks to Hai Long Dao for pointing out to me
that one may use Proposition 3.3 in \cite{dao2013decent} to prove
Proposition~\ref{longresult} below. Finally, I thank the anonymous referee for his or her helpful suggestions.


\section{Kn\"orrer periodicity over $\R$}
\label{ch1}
In this section, we recall some foundational material concerning
matrix factorizations in commutative algebra, and we
exhibit an $8$-periodic version of Kn\"orrer periodicity for matrix
factorization categories associated to isolated hypersurface singularities over
the real numbers.

\subsection{Matrix factorization categories}
We provide some background on matrix factorization categories. Fix a
commutative algebra $Q$ over a field $k$ and an element $f$ of $Q$. Henceforth, when
we use the term ``dg category'', we mean ``$k$-linear differential
$\Z/2\Z$-graded category''.

We cite results on differential $\Z$-graded categories from
\cite{bertrand2011lectures} several times throughout this section; we
refer the reader to Section 5.1 of \cite{dyckerhoff2011compact} for a discussion
as to how one may reformulate the results in \cite{bertrand2011lectures}
so that they apply to the $\Z/2\Z$-graded setting. 


\subsubsection{Definitions and some properties}
\label{mf}

\begin{defn} 
\label{mfdef}
The dg category $\MF(Q,f)$ of \emph{matrix factorizations}
of $f$ over $Q$ is given by the following:

Objects in $\MF(Q,f)$ are pairs $(P,d)$, where $P$ is a
finitely generated projective $\Z/2\Z$-graded $Q$-module, and $d$ is
an odd-degree endomorphism of $P$ such that $d^2=f \cdot
\id_P$. Henceforth, we will often denote an object $(P,d)$ in
$\MF(Q,f)$ by just $P$.

The morphism complex of a pair of matrix factorizations $P, P'$,
which we will denote by \newline $\Hom_\MF(P, P'),$ is the $\Z/2\Z$-graded module of $Q$-linear maps from $P$ to $P'$ equipped with the differential $\partial$ given by
$$\partial(\alpha) = d' \circ \alpha - (-1)^{| \alpha |} \alpha \circ d$$
for homogeneous maps $\alpha: P \to P'$. 
\end{defn}

We will often express an object $P$ in $\MF(Q,f)$ with the notation
$$P_1 \darrow{d_1}{d_0} P_0,$$
where $P_1,P_0$ are the odd and even degree summands of $P$, and
$d_1,d_0$ are the restrictions of $d$ to $P_1$ and $P_0$, respectively.

A degree 0 morphism $\alpha$ in $\MF(Q,f)$ can be represented by a diagram of the following form:
\[
\begin{CD}
P_1 @>{d_1}>> P_0 @> {d_0} >> P_1 \\
@V{\alpha_1}VV  @V{\alpha_0}VV @ VV{\alpha_1}V\\
P_1' @>{d_1'}>> P_0' @> {d_0'} >> P_1'  \\
\end{CD}
\]

It is straightforward to check that $\alpha$ is a cycle if and only
if this diagram commutes. In fact, if $f \in Q$ is a non-zero-divisor,
it is easy to see that the left square commutes if and only if the
right square commutes.

\begin{rem}
\label{samerank}
If $P_1$ and  $P_0$ are free and $f$ is non-zero-divisor, $P_1$ and $P_0$ must have the same rank.
\end{rem}

Define $Z^0\MF(Q, f)$ to be the category with the same objects as
$\MF(Q,f)$ and with morphisms given by the degree 0 cycles in
$\MF(Q,f)$. When $Q$ is regular with finite Krull dimension and $f$ is
a regular element of $Q$ (i.e. $f$ is a non-unit, non-zero-divisor), $Z^0\MF(Q,f)$ is an exact category with the
evident family of exact sequences (\cite{orlov2003triangulated}
Section 3.1).

The \emph{homotopy category}, $[\MF(Q,f)]$, of the dg
category $\MF(Q,f)$ is defined to be the quotient of $Z^0\MF(Q,f)$ by
morphisms that are boundaries in $\MF(Q,f)$. That is, objects in
$[\MF(Q,f)]$ are the same as those of $\MF(Q,f)$, and the morphisms in $[\MF(Q,f)]$
between objects $P, P'$ are classes in $H^0\Hom_\MF(P, P')$.

\begin{defn}
We call a matrix factorization \emph{trivial} if it is a
direct sum of matrix factorizations that are isomorphic in $Z^0\MF(Q,f)$
to either
$$E \darrow{f \cdot \id_E}{\id_E} E$$
or
$$E \darrow{\id_E}{f \cdot \id_E} E$$
for some finitely generated projective $Q$-module $E$.
\end{defn}

The following result gives an alternative characterization for
when a morphism in $Z^0\MF(Q,f)$ is a boundary in $\MF(Q,f)$; the straightforward proof is omitted.

\begin{prop}
\label{trivialmf}
A morphism $\alpha: P \to P'$ in $Z^0\MF(Q,f)$ is a boundary in
$\MF(Q,f)$ if
and only if it factors through a
trivial matrix factorization in $Z^0\MF(Q,f)$.
\end{prop}

We conclude this section with a technical result that will be used in the proof of Proposition~\ref{ABSmap}:

\begin{prop}
\label{cokerfree}
Let $P=(P_1 \darrow{d_1}{d_0} P_0)$ be a matrix
factorization of $f$ over $Q$. Assume $f$ is a non-zero-divisor. Then the following are equivalent:

\begin{itemize}
\item[(1)] $\coker(d_1)$ is isomorphic to $L/fL$ for some projective
  $Q$-module $L$.
\item[(2)] There exists a
trivial matrix factorization $E$ and a matrix factorization $E'$ that
is isomorphic in $Z^0\MF(Q,f)$ to one of the form 
$$F \darrow{\id_{F}}{f} F$$
such that
$P \oplus E'$ is isomorphic to $E$ in $Z^0\MF(Q,f)$.
\end{itemize}
\end{prop}

We will use the following general fact about
idempotent complete categories. We suspect that this result is well-known to experts; we omit the purely formal proof.

\begin{lem}
\label{additive}
Let $\fC$ be an idempotent complete additive category, and let $\fE$ be a
collection of objects in $\fC$ that is 
\begin{itemize}
\item closed under isomorphisms,
\item closed under finite coproducts, and
\item closed under taking summands; that is, whenever $X$ is an object in $\fC$ such that $\id_X$ factors through an
object in $\fE$, $X$ is an object in $\fE$.
\end{itemize}
Denote by $\fL$ the quotient of $\fC$ by those morphisms that factor through an object in
$\fE$. If $X$ and $Y$ are objects in $\fC$, their images in $\fL$ are
isomorphic if and only if there exist objects $E_X$, $E_Y$ in $\fE$ such
that

$$X \oplus E_X \cong Y \oplus E_Y.$$

\end{lem}

We now prove Proposition~\ref{cokerfree}:

\begin{proof}
$(2) \Rightarrow (1)$: Since the cokernel of $d_1$ is isomorphic to
the
cokernel of 
$$d_1 \oplus \id_{F}: P_1 \oplus F \to P_0 \oplus F,$$ we
may assume $P$ is trivial. In this case, the result is obvious.

$(1) \Rightarrow (2)$: Choose a projective $Q$-module $L$ such that there exists an isomorphism $$\coker(d_1) \xrightarrow{\cong} L/fL.$$ We have
$Q$-projective resolutions

$$0 \to P_1 \xrightarrow{d_1}  P_0 \to \coker(d_1) \to 0$$
$$0 \to L \xrightarrow{f} L \to L/fL \to 0$$

Thus, there exist maps
$$\beta_i: P_i \to L \text{, } \gamma_i: L \to P_i$$
for $i=0,1$ making the following diagrams commute:
\[
\begin{CD}
0 @>>> P_1 @>{d_1}>> P_0 @> {} >> \coker(d_1)  @>>> 0\\
@. @V{\beta_1}VV  @V{\beta_0}VV @ VV{\cong}V @.\\
0 @>>>L  @>{f}>> L  @> {} >> L/fL  @>>>0 \\
\end{CD}
\]
\[
\begin{CD}
0 @>>>L  @>{f}>>L  @> {} >> L/fL  @>>>0 \\
@. @V{\gamma_1}VV  @V{\gamma_0}VV @ VV{\cong}V @.\\
0 @>>> P_1 @>{d_1}>> P_0 @> {} >> \coker(d_1) @>>>0 \\
\end{CD}
\]

Hence, we have maps
$$h_P: P_0 \to P_1 \text{, } h_{L}: L \to L$$
such that
$$ \gamma_1 \circ \beta_1 - \id_{P_1} = h_P \circ d_1 \text{, } \gamma_0 \circ \beta_0 - \id_{P_0}= d_1 \circ h_P.$$
$$ \beta_1 \circ \gamma_1 - \id_{L} = fh_{L} \text{, } \beta_0 \circ \gamma_0 - \id_{L} = fh_{L}.$$

We have commutative diagrams
$$
\xymatrix{
P_1 \ar[rr]^-{d_1} \ar[dd]^-{h_P \circ d_1} & & P_0 \ar[dd]^-{h_P}   & L  \ar[rr]^-{f \cdot \id_L} \ar[dd]^-{f \cdot h_L}& &L \ar[dd]^-{h_L}\\
& & & \\
P_1 \ar[rr]^-{\id_{P_1}} \ar[dd]^-{\id_{P_1}} & &P_1 \ar[dd]^-{d_1}    & L \ar[rr]^-{\id_{L}} \ar[dd]^-{\id_{L}} & & L  \ar[dd]^-{f \cdot \id_L}  \\
& & & \\
P_1 \ar[rr]^-{d_1}& &P_0    & L \ar[rr]^-{f \cdot \id_{L}}& & L \\
}
$$

Denote by $\fE$ the collection
of matrix factorizations of $f$ over $Q$ isomorphic in $Z^0\MF(Q,f)$ to
a matrix factorization of the form
$$E \darrow{\id_E}{f} E.$$

Notice that $Z^0\MF(Q,f)$ is an idempotent complete additive category,
and $\fE$ is closed under direct sums and direct summands in
$Z^0\MF(Q,f)$. Letting $\fL$ denote the quotient of $Z^0\MF(Q,f)$ by those
morphisms that factor through an object in $\fE$, we have that 
$$(P_1
\darrow{d_1}{d_0} P_0) \cong (L\darrow{f}{\id_{L}} L)$$
in $\fL$. The
result now follows from Lemma~\ref{additive}.

\end{proof}


\subsubsection{Triangulated structure}
\label{tristruc}

Suppose $Q$ is regular with finite Krull dimension and $f$ is a
regular element of $Q$. A feature of the homotopy category $[\MF(Q,f)]$ is that it may be equipped with a triangulated
structure in the following way (\cite{orlov2003triangulated} Section 3.1):

The shift functor maps the object
$$P=(P_1 \darrow{d_1}{d_0} P_0)$$
to the object
$$P[1]=(P_0 \darrow{-d_0}{-d_1} P_1).$$

Given a morphism $\alpha: (P_1 \darrow{d_1}{d_0} P_0) \to (P_1' \darrow{d_1'}{d_0'} P_0')$ in $Z^0\MF(Q,f)$, the \emph{mapping cone} of $\alpha$ is defined as follows:
$$\cone(\alpha)=(P_0' \oplus P_1 \darrow{\begin{pmatrix} d_0' & \alpha_1 \\ 0 &
-d_1\end{pmatrix}}{\begin{pmatrix} d_1' & \alpha_0 \\ 0 &
-d_0\end{pmatrix}} P_1' \oplus P_0)$$

There are canonical morphisms $P' \to \cone(\alpha)$ and
$\cone(\alpha) \to P[1]$ in $Z^0\MF(Q,f)$. Taking the distinguished
triangles in $[\MF(Q,f)]$ to be the triangles isomorphic in
$[\MF(Q,f)]$ to those of the form 
$$P \xrightarrow{\alpha}  P' \to \cone(\alpha) \to P[1],$$ $[\MF(Q,f)]$ may be equipped with the structure of a triangulated
category.

The Grothendieck group, $K_0[\MF(Q,f)]$, of the triangulated category $[\MF(Q,f)]$ is defined to be the free abelian group generated by isomorphism classes of $[\MF(Q,f)]$ modulo elements of the form $[P_1] - [P_2] +[P_3]$, where $P_1, P_2,$ and $P_3$ fit into a distinguished triangle in the following way:
$$P_1 \to P_2 \to  P_3 \to P_1[1].$$

\begin{rem}
The category $\MF(Q,f)$ is not always triangulated in the dg sense (see Section 4.4 of \cite{bertrand2011lectures} for the
definition of a triangulated dg category). When $\MF(Q, f)$ is triangulated in
the dg sense,
the induced triangulated structure on $[\MF(Q, f)]$ agrees with the
triangulated structure just described. 
\end{rem}

\begin{rem} 
\label{MCM} 
When $Q$ is a regular local ring and $f$ is a regular
  element of $Q$, one has an equivalence of triangulated categories
$$[\MF(Q,f)] \xrightarrow{\cong} \underline{\MCM}(Q/(f)),$$
where $\underline{\MCM}(Q/(f))$ denotes the \emph{stable category} of
maximal Cohen-Macaulay (MCM) modules over the ring $Q/(f)$. The stable
category of MCM modules is obtained by taking the quotient of the
category of MCM modules over $Q/(f)$ by those morphisms that factor
through a projective $Q/(f)$-module. The above equivalence
is given, on objects, by
$$(P_1 \darrow{d_1}{d_0} P_0) \mapsto \coker(d_1).$$

Matrix factorizations were first defined by Eisenbud in
\cite{eisenbud1980homological}; this interplay between matrix
factorizations and MCM modules over hypersurface rings provided the original motivation for the
study of matrix factorization categories.
\end{rem}


\subsubsection{Stabilization}
\label{stabilization}
Assume now that $Q$ is a regular local ring of Krull dimension $n$, and suppose $f$ is a regular element of
$Q$. Denote by $\DD(Q/(f))$ the bounded derived category of
$Q/(f)$, and set $\Perf(Q/(f))$ to be the
full subcategory of $\DD(Q/(f))$ given by perfect complexes. $\Perf(Q/(f))$ is a thick subcategory of
$\DD(Q/(f))$; define $\underline{\DD}(Q/(f))$ to be the Verdier
quotient of  $\DD(Q/(f))$ by $\Perf(Q/(f))$. In
\cite{buchweitz1986maximal}, Buchweitz
calls this quotient the \emph{stabilized derived category} of $Q/(f)$.

By \cite{buchweitz1986maximal}, the functor
$$\underline{\MCM}(Q/(f)) \to \underline{\DD}(Q/(f))$$
that sends an MCM module $M$ to the complex with $M$ concentrated in
degree 0 is a triangulated equivalence. Hence, composing with the
equivalence in Remark~\ref{MCM}, one has an equivalence
$$[\MF(Q, f)] \to \underline{\DD}(Q/(f))$$

Following \cite{dyckerhoff2011compact}, given an object $C$ in $\underline{\DD}(Q/(f))$,
we denote by $C^{\stab}$ the isomorphism class in $[\MF(Q,f)]$
corresponding to $C$ under the above equivalence
(``stab''  stands for ``stabilization''). In particular, thinking of the residue field $k$ of $Q/(f)$ as a
complex concentrated in degree 0, we may associate to $k$ an
isomorphism class
$k^{\stab}$ in $[\MF(Q, f)]$. We now construct an object $E_f$ in $\MF(Q,
f)$ that represents $k^{\stab}$; this construction appears in \cite{dyckerhoff2011compact}. Choose a regular system of parameters $x_1, \dots, x_n$ for $Q$, and consider
the Koszul complex
$$(\bigoplus_{i=0}^n \largewedge^i Q^n, s_0)$$
as a $\Z/2\Z$-graded complex of free $Q$-modules with
even (odd) degree piece given by the direct sum of the even (odd)
exterior powers of $Q^n$. Here, $s_0$ denotes the
$\Z/2\Z$-folding of the Koszul differential associated to $x_1, \dots,
x_n$.
Choose an expression of $f \in Q$ of the form
$$f=g_1x_1 + \dots + g_nx_n.$$
Fix a basis $e_1, \dots, e_n$ of $Q^n$, and set $s_1$ to be the
odd-degree endomorphism of $\bigoplus_{i=0}^n \largewedge^i Q^n$ given by exterior multiplication on
the left by $g_1e_1+\dots+g_ne_n$.
Define
$$E_f:=(\bigoplus_{i=0}^n \largewedge^i Q^n, s_0+s_1).$$

It is easy to check that $E_f$ is a matrix factorization of $f$. By
Corollary 2.7 in \cite{dyckerhoff2011compact}, $E_f$ represents $k^{\stab}$ in $[\MF(Q, f)]$.
In particular, $E_f$ does not depend on the choice of regular system of
parameters $x_1, \dots, x_n$ or coefficients $g_1, \dots,
g_n$ up to homotopy equivalence. Henceforth, we shall denote the dg algebra
$\End_{\MF}(E_f)$ by $A_{(Q,f)}$.


\subsection{The tensor product of matrix factorizations}
\label{tsmf}
Let $k$ be a field. We begin this section with a technical definition:
\begin{defn} 
\label{IHS}
Suppose $Q$ is a commutative algebra over $k$ and $f \in Q$. If the pair $(Q, f)$ satisfies
\begin{itemize}
\item $Q$ is essentially of finite type over $k$
\item $Q$ is equidimensional of dimension $n$
\item The module $\Omega^1_{Q/k}$ of K\"ahler differentials is locally
  free of rank $n$
\item The zero locus of $df \in \Omega^1_{Q/k}$ is a
  0-dimensional scheme supported on a unique closed point $\m$ of
  $\Spec(Q)$ with residue field $k$ and $f \in \m$
\end{itemize}
we shall call the pair $Q/(f)$ an \emph{isolated hypersurface singularity}, or
$\IHS$.
\end{defn} 

\begin{rem}
Our $\IHS$ condition above is precisely condition (B) in Section 3.2
of \cite{dyckerhoff2011compact}. As noted in loc. cit., if $Q/(f)$ and $Q'/(f')$ are $\IHS$, $Q \otimes_k Q' / (f \otimes 1 + 1 \otimes f')$
is as well.
\end{rem}
Suppose $Q$ and $Q'$ are commutative algebras over $k$, $f \in Q$, and $f' \in Q'$. Given objects $P$ and $P'$ in $\MF(Q,f)$, $\MF(Q', f')$, one can form their \emph{tensor product} over $k$:
$$P \otimes_{\MF} P':=((P_1 \otimes_k P_0') \oplus (P_0 \otimes_k P_1')  \darrow{\begin{pmatrix} d_1 \otimes \id_{P_0'} & \id_{P_0} \otimes d_1'  \\ -\id_{P_1} \otimes d_0' &
d_0 \otimes \id_{P_1'}\end{pmatrix}}{\begin{pmatrix} d_0 \otimes \id_{P_0'} & -\id_{P_1} \otimes d_1'  \\ \id_{P_0} \otimes d_0' &
d_1 \otimes \id_{P_1'}\end{pmatrix}} (P_0 \otimes_k P_0') \oplus (P_1 \otimes_k P_1')).$$

This construction first appeared in \cite{yoshino1998tensor}; it can
be thought of as a $\Z/2\Z$-graded analogue of the tensor product of
complexes. It is straightforward to check that $P \otimes_{\MF} P'$ is
an object in $\MF(Q \otimes_k Q', f \otimes 1 + 1 \otimes f')$. In fact, setting $f \oplus f':=f \otimes 1 + 1 \otimes f' \in Q
\otimes_k Q'$, and noting that there is a canonical map of complexes
$$\Hom_{\MF}(P, L) \otimes_k \Hom_{\MF}(P', L') \to \Hom_{\MF}(P \otimes_{\MF} P', L \otimes_{\MF} L'),$$
we have the following:

\begin{prop}
\label{littleTS}
There is a dg functor
$$\ST_{\MF}: \MF(Q, f) \otimes_k \MF(Q', f') \to \MF(Q \otimes_k Q', f \oplus f')$$
that sends an object $(P, P')$ to $P \otimes_{\MF} P'$.
\end{prop}

\begin{rem}
\label{K0pairing}
It is straightforward to verify that $\ST_\MF$ induces a pairing
$$K_0[\MF(Q,f)] \otimes K_0[\MF(Q', f')] \to K_0[\MF(Q \otimes_k Q', f
\oplus f')].$$
\end{rem}

\begin{rem}
\label{reasonforname1}
The ``ST'' in the name $\ST_\MF$ stands for ``Sebastiani-Thom'', since 
this tensor product operation is related to the Sebastiani-Thom homotopy
equivalence discussed in Section~\ref{TStopp}. A precise sense in
which the tensor product of matrix factorizations is related to the
Sebastiani-Thom homotopy equivalence is illustrated by the proof of
Proposition~\ref{compatible} below; see Remark~\ref{reasonforname} for
further details.
\end{rem}

Now, suppose $Q/(f)$ and $Q'/(f')$ are $\IHS$. Set $Q'':=Q \otimes_k Q'$. We will denote by $\wQ$ the $\m$-adic
completion of $Q_\m$, where $\m$ is as in the definition of $\IHS$. Define $\widehat{Q'}$ and $\widehat{Q''}$ similarly, and define
$$\phi: \widehat{Q} \otimes_k \widehat{Q'} \to \widehat{Q''}$$
to be the canonical ring homomorphism. $\phi$ induces a dg functor
$$\MF(\phi): \MF(\wQ \otimes_k
\widehat{Q'}, f \oplus f') \to \MF(\widehat{Q''}, f \oplus f').$$
Set $\widehat{\ST}_{\MF}$ to be the composition of $\MF(\phi)$ with the tensor product functor
$$\MF(\wQ, f) \otimes_k \MF(\widehat{Q'}, f') \to \MF(\wQ \otimes_k \widehat{Q'}, f \oplus f').$$

\begin{prop}
\label{TS}
If $Q/(f)$ and $Q'/(f')$ are $\IHS$,
$$\widehat{\ST}_{\MF}: \MF(\wQ, f) \otimes_k \MF(\widehat{Q'}, f') \to
\MF(\widehat{Q''}, f\oplus f')$$

is a Morita equivalence of dg categories.
\end{prop}

\begin{rem} We emphasize that Proposition~\ref{TS} is really a straightforward application of several results
in \cite{dyckerhoff2011compact}; we include a proof for completeness. We refer the reader to Section 4.4 of \cite{bertrand2011lectures} for the definition of a Morita
equivalence of dg categories.
\end{rem}

\begin{proof}
Let $\m$ and $\m'$ be the maximal ideals of $Q$ and $Q'$ arising in
Definition~\ref{IHS}. Suppose $Q_\m$ and $Q'_{\m'}$ have Krull dimensions $n$ and $m$,
respectively. $Q_\m$ and $Q'_{\m'}$ are regular local rings; choose regular systems of parameters $x_1, \dots, x_n$ and $y_1,
\dots, y_m$ in $Q_\m$ and $Q'_{\m'}$, and choose expressions
$$f=g_1x_1 + \dots + g_nx_n$$
$$f'=h_1y_1 + \dots + h_my_m$$
of $f$ and $f'$. Use these expressions to construct the dga's $A_{(Q_\m,
  f)}$ and $A_{(Q'_{\m'}, f')}$, as in Section~\ref{stabilization}.

Note that $x_1, \dots, x_n$ and $y_1, \dots, y_m$ form regular
systems of parameters in
$\wQ$ and $\widehat{Q'}$ as well, so we may use these expressions to
construct $A_{(\widehat{Q}, f)}$ and $A_{(\widehat{Q'}, f')}$. Also, $x_1 \otimes 1, \dots, x_n \otimes 1, 1 \otimes y_1, \dots, 1 \otimes
y_m$ is a regular system of parameters in $Q''_{\m''}$, where
$\m'':=\m \otimes 1 + 1 \otimes \m'$, so we may
use the expression
$$f \oplus f' = (g_1x_1 \otimes 1) +  \dots + (g_nx_n \otimes 1) + (1 \otimes
h_1y_1) +\dots + (1 \otimes
h_my_m)$$
to construct $A_{(Q''_{\m''}, f \oplus f')}$ and $A_{(\widehat{Q''} , f \oplus f')}$.

By Section 6.1 of \cite{dyckerhoff2011compact}, we have a quasi-isomorphism
$$F: A_{(Q_{\m}, f)} \otimes_k A_{(Q'_{\m'}, f')} \xrightarrow{\cong}
A_{(Q''_{\m''} , f \oplus f')}.$$
We also have a canonical map
$$G: A_{(\wQ, f)} \otimes_k A_{(\widehat{Q'}, f')} \to A_{(\widehat{Q''} , f \oplus f')}.$$
By the proof of Theorem 5.7 in \cite{dyckerhoff2011compact}, the inclusions
$$A_{(Q_\m, f)}  \into A_{(\wQ, f)}$$
$$A_{(Q'_{\m'}, f')} \into A_{(\widehat{Q'}, f')}$$
$$A_{(Q''_{\m''}, f \oplus f')} \into A_{(\widehat{Q''}, f \oplus f')}$$
are all quasi-isomorphisms. Since a tensor product of Morita equivalences is again a Morita
equivalence (\cite{bertrand2011lectures} Section 4.4), it follows that the
map
$$A_{(Q_\m, f)} \otimes^{}_k A_{(Q'_{\m'}, f')} \to A_{(\wQ, f)}
\otimes^{}_k A_{(\widehat{Q'}, f')}$$
is a Morita equivalence.

We have the following commutative square:
$$
\xymatrix{
A_{(Q_\m, f)} \otimes^{}_k A_{(Q'_{\m'}, f')} \ar[r] \ar[d] & A_{(\wQ, f)} \otimes^{}_k A_{(\widehat{Q'}, f')} \ar[d]^-{G} \\
A_{(Q''_{\m''}, f \oplus f')} \ar[r] & A_{(\widehat{Q''}, f\oplus f')} \\
}
$$

It follows that $G$ is a Morita equivalence.

One may think of a dga as a dg category with a
single object. Adopting this point of view, we have inclusion functors 
$$i: A_{(\wQ, f)} \into \MF(\wQ,f)$$
$$j: A_{(\widehat{Q'},
  f')} \into \MF(\widehat{Q'},f')$$
$$l: A_{(\widehat{Q''}, f\oplus f')} \into \MF(\widehat{Q''},f \oplus
f')$$
Combining Theorem 5.2 and
Lemma 5.6 in
\cite{dyckerhoff2011compact}, we conclude that $i$, $j$, and $l$ are
Morita equivalences. In particular, we have that
$$i \otimes j: A_{(\wQ, f)} \otimes^{}_k A_{(\widehat{Q'},
  f')} \to \MF(\wQ,f)
\otimes^{}_k \MF(\widehat{Q'} , f')$$ 
is a Morita equivalence.

Finally, consider the following commutative diagram:
$$
\xymatrix{
A_{(\wQ, f)} \otimes^{}_k A_{(\widehat{Q'},
  f')} \ar[r]^-{i \otimes j} \ar[d]^{G} & \MF(\wQ,f)
\otimes^{}_k \MF(\widehat{Q'} , f') \ar[d]^-{\wST_{\MF}} \\
A_{(\widehat{Q''}, f\oplus f')} \ar[r]^-{l} & \MF( \widehat{Q''} ,f \oplus f') \\
}
$$

Since the left-most vertical map and both horizontal maps are
Morita equivalences, $\wST_{\MF}$ is as well.

\end{proof}

\begin{rem}
\label{polyTS}
Using Theorem 4.11 of \cite{dyckerhoff2011compact} along with a similar argument to the one above, one
may show that, under the assumptions of Proposition~\ref{TS}, the functor
$$\MF(Q, f) \otimes_k \MF(Q', f') \to \MF(Q'', f \oplus f')$$
given by tensor product of matrix factorizations is also a Morita
equivalence. 
\end{rem}


\subsection{Matrix factorizations of quadratics}
\label{clifford}
Fix a field $k$ such that $\chara(k) \ne 2$ and a finite-dimensional vector
space $V$ over $k$. Let $q:V \to k$ be a 
quadratic form, and let $\cliff_k(q)$ denote the Clifford algebra associated to
$q$. $\cliff_k(q)$ is a $\Z/2\Z$-graded $k$-algebra; let $\fgmod_{\Z/2\Z}(\cliff_k(q))$ denote the category of finitely generated
$\Z/2\Z$-graded left modules over $\cliff_k(q)$. Henceforth, when we
refer to a module over a Clifford algebra, we will always mean it to
be a left module.

Assume $q$ is non-degenerate, and choose a basis $\{e_1, \dots,
e_n\}$ of $V$ with respect to which $q$ is \emph{diagonal}; that is,
$$q=a_1x_1^2 + \dots + a_nx_n^2 \in S^2(V^*)$$
where the $x_i$ comprise the dual basis corresponding to the $e_i$,
and the $a_i$ are nonzero elements of $k$. Denote by $Q$ the
localization of $S(V^*)$ at the ideal $(x_1, \dots, x_n)$.

The following theorem, due to Buchweitz-Eisenbud-Herzog, yields a relationship between Clifford modules
and matrix factorizations of non-degenerate quadratic forms:

\begin{thm}[\cite{buchweitz1987cohen}]
\label{BEH}
There is an equivalence of $k$-linear categories
$$\fgmod_{\Z/2\Z}(\cliff_k(q)) \xrightarrow{\cong} [\MF(\wQ, q)].$$ 
\end{thm}

Denote by $\Theta$ the explicit construction of this equivalence described in
the proof of Theorem 14.7 in
\cite{yoshino1990maximal}.

\begin{rem}
\label{powerseries}
The inclusion
$$k[x_1, \dots, x_n] \into \widehat{Q}$$
induces an equivalence
$$[\MF(k[x_1, \dots, x_n], q_n)] \xrightarrow{\cong} [\MF(\widehat{Q}, q_n)].$$

To see this, we first recall that every matrix
factorization of $q_n$ over $\wQ$ is isomorphic in
$[\MF(\widehat{Q}, q_n)]$ to one with (linear) polynomial entries
(\cite{yoshino1990maximal} Proposition 14.3); hence, the functor is essentially
surjective. 

Also, one has a commutative diagram
$$
\xymatrix{
{[\MF(Q,  q_n)]} \ar[r]^-{\cong} \ar[d] & \underline{\MCM}(Q/(q_n)) \ar[d] \\
{[\MF(\wQ,  q_n)]}  \ar[r]^-{\cong} & \underline{\MCM}(\wQ/(q_n)) \\
}
$$

The morphism sets in $\underline{\MCM}(Q/(q_n))$
are Artinian modules, and hence complete. Thus, the functor on the right is fully
faithful, and so the functor on the left is as well. 

It now follows from Theorem 4.11 in \cite{dyckerhoff2011compact} that the functor
$$[\MF(k[x_1, \dots, x_n], q_n)] \to [\MF(\widehat{Q}, q_n)]$$
is fully faithful.
\end{rem}

Suppose $q': W \to k$ is another non-degenerate quadratic
form; choose a basis of $W$ with respect to which $q'$ is diagonal,
and let $y_1, \dots, y_m$ denote the corresponding basis of $W^*$. As above, we may think of $q'$ as an element
of $S^2(W^*)$. Set
$Q'$ to be the localization of $S(W^*)$ at the ideal $(y_1, \dots, y_m)$.

It is well-known that the $\Z/2\Z$-graded tensor product of
$\cliff_k(q)$ and $\cliff_k(q')$ over $k$ is canonically isomorphic to
$\cliff_k(q \oplus q')$. Further, by Remark 1.3 in
\cite{yoshino1998tensor}, the $\Z/2\Z$-graded tensor product of
Clifford modules is compatible, via this canonical isomorphism and the
equivalence in Theorem~\ref{BEH}, with the tensor product $\ST_\MF$ in
Proposition~\ref{littleTS}. That is, one has a commutative diagram of
$k$-linear categories

$$
\xymatrix{
\fgmod_{\Z/2\Z}(\cliff_k(q)) \times \fgmod_{\Z/2\Z}(\cliff_k(q')) \ar[r] \ar[d]^-{\Theta \times \Theta} & \fgmod_{\Z/2\Z}(\cliff_k(q \oplus q' )) \ar[d]^-{\Theta} \\
{[\MF(Q, q)] \times [\MF(Q', q')]} \ar[r]^-{\ST_\MF} & {[\MF(Q \otimes_k Q', q \oplus q')]}
}
$$

Let $C$ be a rank 1 free $\Z/2\Z$-graded $\cliff_k(q)$-module. If
$\dim(V) = 1$ and $q = x^2$, it is easy to check that the isomorphism
class of $\Theta(C)$ is
$k^{\stab}$, where $k^{\stab}$ is as defined in
Section~\ref{stabilization}. Further, $E_q \otimes_\MF E_{q'} \cong E_{q \oplus q'}$, where $E_q$, $E_{q'}$, and $E_{q \oplus q'}$ are as in Section~\ref{stabilization} (\cite{dyckerhoff2011compact} Section 6.1). Thus, we have:

\begin{prop}
\label{cliffstab}
If $a_i = 1$ for $1 \le i \le n$, the isomorphism class of
$\Theta(C)$ is $k^{\stab}$.
\end{prop}


\subsection{Periodicity}
\label{periodicity}
Following \cite{dyckerhoff2011compact}, given a commutative algebra $Q$ over a field $k$ and an element $f$ of $Q$, we define $\MF^{\infty}(Q, f)$ to be the dg category of possibly infinitely-generated matrix factorizations; that is, objects of $\MF^{\infty}(Q,f)$ are defined in the same way as $\MF(Q,f)$, except the projective $\Z/2\Z$-graded $Q$-module $P$ need not be finitely generated. 

A version of Kn\"orrer
periodicity (Theorem~\ref{kpint}) for isolated hypersurface singularities may be deduced from the following proposition:

\begin{prop}
\label{k}
Suppose $Q$ and $Q'$ are commutative algebras over a field $k$. Let
$f \in Q$ and $f' \in Q'$, and suppose $Q/(f)$ and $Q'/(f')$ are $\IHS$. If there exists an object $X$ in
$\MF(Q', f')$ such that
\begin{itemize}
\item[(a)] $X$ is a compact generator of $[\MF^{\infty}(Q', f')]$,
  and 
\item[(b)] the inclusion $k \into \End_{\MF(\widehat{Q'}, f')}(X)$ is a
  quasi-isomorphism
\end{itemize}
then the dg functor
$$K_X: \MF(\wQ, f) \to \MF(\widehat{Q \otimes_k Q'}, f \oplus f')$$
given by 
$$P \mapsto P \otimes_\MF X$$ 
on objects and 
$$\alpha \mapsto \alpha \otimes \id_X$$ 
on morphisms is a quasi-equivalence.
\end{prop}

\begin{proof}
By Theorems 4.11, 5.1, and 5.7 in \cite{dyckerhoff2011compact}, the inclusion
$$\End_{\MF(\widehat{Q'}, f')}(X) \hookrightarrow \MF(\widehat{Q'}, f')$$
is a Morita equivalence. We have a
chain of Morita equivalences
$$\MF(\wQ, f) \otimes_k k \hookrightarrow \MF(\wQ, f) \otimes_k  \End_{\MF(\widehat{Q'}, f')}(X) \hookrightarrow \MF(\wQ, f) \otimes_k\MF(\widehat{Q'},
f').$$
Composing with $\widehat{\ST}_\MF$, Proposition~\ref{TS} yields a
Morita equivalence 
$$\MF(\wQ, f) \to \MF(\widehat{Q \otimes_k Q'}, f \oplus f').$$

This composition is clearly the functor $K_X$; thus, $K_X$ is a Morita
equivalence. Since both $\MF(\wQ,
f)$ and $\MF(\widehat{Q \otimes_k Q'}, f \oplus f')$ are triangulated in the dg sense by Lemma
5.6 in \cite{dyckerhoff2011compact}, we may apply Theorem 3.2.1 in
\cite{bertrand2011lectures} and Theorem 1.2.10 in
\cite{hovey2007model} to conclude that $K_X$ is a quasi-equivalence.
\end{proof}

To deduce a version of Kn\"orrer periodicity for isolated hypersurface
singularities,
assume $k$ to be an algebraically closed field such that $\chara(k)
\ne 2$, set $Q'=k[u,v]$ and $f'=u^2+v^2$, and take $X$ to be the matrix factorization 
$$k[u,v] \darrow{u+iv}{u-iv} k[u,v].$$ 
This is the approach taken in Section 5.3 of
\cite{dyckerhoff2011compact}.

We point out that $k$ is not assumed to be algebraically closed in
Proposition~\ref{k}, and no assumptions on the characteristic of $k$
are made, either. In particular, we may use Proposition~\ref{k}
to prove an $8$-periodic version of Kn\"orrer periodicity over $\R$
(this result implies Theorem~\ref{thm1} from the introduction):

\begin{thm}
\label{rk}
Suppose $Q$ is an $\R$-algebra. Let $f \in Q$, and suppose $Q/(f)$ is
$\IHS$. Set $Q':=\R[u_1, \dots, u_8]$. Then there exists a
matrix factorization $X$
of $-u_1^2 - \cdots - u_8^2$ over $Q'$ such that the dg functor
$$\MF(\wQ,f) \to
\MF(\widehat{Q \otimes_\R Q'},f - u_1^2 - \cdots - u_8^2)$$
given by 
$$P \mapsto P \otimes_\MF X$$ 
on objects and 
$$\alpha \mapsto \alpha \otimes \id_X$$ 
on morphisms is a quasi-equivalence.
\end{thm}

\begin{rem}
One may replace $-u_1^2 - \dots - u_8^2$ with $u_1^2 + \cdots + u_8^2$
and obtain a similar result; the proof is the same.
\end{rem}

\begin{proof}
Set $q:=-u_1^2 - \cdots - u_8^2 \in Q'$. We equip the matrix algebra $\Mat_{16}(\R)$ of $16 \times 16$ of matrices
over $\R$ with a $\Z/2\Z$-grading in the following way:
$A=(a_{ij})$ is
homogeneous of even degree if $a_{ij}=0$ whenever $i +j$ is odd, and
$A$ is homogeneous of odd degree if $a_{ij}=0$ whenever $i +j$ is
even. By Proposition V.4.2 in \cite{lam2005introduction},
$$\cliff_\R(q) \cong \Mat_{16}(\R)$$
as $\Z/2\Z$-graded algebras. In particular, by Theorem~\ref{BEH},
$$[\MF(\widehat{Q'}, q)] \cong \fgmod_{\Z/2\Z}(\Mat_{16}(\R)),$$
where the right hand side is the category of finitely generated $\Z/2\Z$-graded left $\Mat_{16}(\R)$-modules. Let $M \in \fgmod_{\Z/2\Z}(\Mat_{16}(\R))$ be the module consisting of elements of $\Mat_{16}(\R)$ with nonzero entries only in the first column. Recall that, by Remark~\ref{powerseries}, the canonical map
$$[\MF(Q', q)] \to [\MF(\widehat{Q'}, q)]$$
is an equivalence; let $X$ be an object of $[\MF(Q', q)] $ corresponding to $M$. 

Let $\m:=(u_1, \dots, u_8) \subseteq Q'$, and let $E_q \in \MF(Q'_{\m}, q)$ be as in Section~\ref{stabilization}. Notice that, by Proposition~\ref{cliffstab},
$(X \oplus X[1])^{\oplus 8} \cong E_q$
in $[\MF(Q'_{\m}, q)]$. In particular, it follows
from Theorems 4.1 and 4.11 of \cite{dyckerhoff2011compact} that $X$ is a compact generator of
$[\MF^{\infty}(Q',
q)]$.

Since $\End_{\Mat_{16}(\R)}(M) \cong \R$ as $\Z/2\Z$-graded $\R$-algebras,
where $\R$ is concentrated in even degree, we have
$H^0(\End_\MF(X)) \cong \R$. We now show $H^1(\End_\MF(X)) = 0$.  By Section 5.5 of \cite{dyckerhoff2011compact}, $H^0(\End_\MF(E_q)) \oplus H^1(\End_\MF(E_q)) $ is isomorphic, as a $\Z/2\Z$-graded $\R$-vector space, to $\cliff_\R(q)$, and so $H^1(\End_\MF(E_q))$ has rank $128$. Also, we have isomorphisms
$$H^1(\End_\MF(E_q)) \cong H^1(\End_\MF((X \oplus X[1])^{\oplus 8})) \cong H^0(\End_\MF(X))^{128} \oplus H^1(\End_\MF(X))^{128}.$$
Thus, $H^1(\End_\MF(X)) = 0$, and so the inclusion
$$\R \into \End_\MF(X)$$
is a quasi-isomorphism. Now apply Proposition~\ref{k}.
\end{proof}

\begin{rem}
Theorem~\ref{rk} implies the existence of a Kn\"orrer-type periodicity for matrix
factorizations over $\R$ of period at most $8$. We point out that the
period is exactly $8$, since the Brauer-Wall group of $\R$ is the
cyclic group $\Z/8\Z$ generated by the class of $\cliff_\R(x^2)$.
\end{rem}


\section{Matrix factorizations and the $K$-theory of the Milnor fiber}
\label{ch2}
We have demonstrated that matrix
factorization categories exhibit 2- and 8-periodic versions of Kn\"orrer
periodicity over $\C$ and $\R$, respectively. This pattern resembles Bott periodicity in
topological $K$-theory; the goal of this section is to explain this
resemblance. 

We give a rough sketch of our approach. The classical link between the periodicity of Clifford algebras up to
$\Z/2\Z$-graded Morita equivalence and Bott periodicity in topological
$K$-theory is the \emph{Atiyah-Bott-Shapiro construction}, which first
appeared in Part III of \cite{atiyah1964clifford} (and, in fact, a \emph{proof} of Bott periodicity using Clifford algebras is provided by Wood in \cite{wood1966banach}). Loosely speaking,
the Atiyah-Bott-Shapiro construction is a way of mapping a finitely generated $\Z/2\Z$-graded module over a real or complex Clifford
algebra to a class in the $K$-theory of a sphere.

Composing the Buchweitz-Eisenbud-Herzog equivalence
(Theorem~\ref{BEH}) with the Atiyah-Bott-Shapiro construction, we have a
way of assigning a class in the topological $K$-theory of a sphere to
a matrix factorization of a non-degenerate quadratic form over $\R$ or
$\C$:
$$\begin{tikzpicture}[scale=1.5]
\node (A) at (0,1) {$\text{mf's of real/complex quadratics}$};
\node (B) at (5,1) {$\text{$K$-theory of spheres}$};
\path[->, font=\scriptsize,>=angle 90]
(A) edge node[above]{$\ABS \circ \BEH$} (B);
\end{tikzpicture}$$

The idea is to lift this composition; that is, we wish to associate a
space $X_f$ to a real or complex polynomial $f$ and construct a map
from matrix factorizations of $f$ to the topological $K$-theory of $X_f$
so that the diagram

$$\begin{tikzpicture}[scale=1.5]
\node (A) at (0,1) {$\text{mf's of real/complex quadratics}$};
\node (B) at (5,1) {$\text{$K$-theory of spheres}$};
\node (C) at (0,0) {$\text{mf's of real/complex polynomials}$};
\node (D) at (5,0) {$\text{$K$-theory of spaces of the form}$ $X_f$};
\path[->, font=\scriptsize,>=angle 90]
(A) edge node[above]{$\ABS \circ \BEH$} (B);
\path[->, dotted, font=\scriptsize,>=angle 90]
(C) edge node[above]{} (D);
\path[->, right hook-latex, font=\scriptsize,>=angle 90]
(A) edge node[right]{} (C)
(B) edge node[right]{} (D);
\end{tikzpicture}$$
commutes.

It turns out that the right choice of $X_f$ is the \emph{Milnor fiber}
(\emph{positive or negative Milnor fiber}) associated to
the complex (real) polynomial. 

We begin this section with discussions of known results concerning the
Milnor fiber and relative topological $K$-theory. Then, using the work of Atiyah-Bott-Shapiro in \cite{atiyah1964clifford} as
a guide, we will complete the above diagram, and we will use the bottom arrow to
explain a precise sense in which Kn\"orrer periodicity and Bott
periodicity are compatible phenomena.


\subsection{The real and complex Milnor fibers}
\label{milnorfiber}

Let $f \in \C[x_1, \dots, x_n]$, and suppose $f(0) = 0$. We begin this
section by
describing the
construction of the Milnor fiber associated to $f$, following the exposition in
Section 1 of \cite{buchweitz2012index}. We then discuss various
properties of the Milnor fiber that we will make use of later on.

\subsubsection{Construction of the Milnor fibration and some properties of
the Milnor fiber}

For $\epsilon > 0$, define $B_\epsilon$ to be the closed ball centered
at the origin of radius
$\epsilon$ in $\C^{n}$, and for $\delta > 0$, set $D_\delta^*$ to be
the open punctured disk centered at the origin in $\C$ of radius $\delta$.

Choose $\epsilon>0$ so that, for $0 < \epsilon' \le \epsilon$, 
$\partial B_{\epsilon'}$ intersects $f^{-1}(0)$ transversely. Upon choosing
such a number $\epsilon$, choose $\delta \in (0, \epsilon)$ such that $f^{-1}(t)$ intersects
$\partial B_\epsilon$
  transversely for all $t \in D_\delta^*$. Then the map
$$\psi: B_\epsilon \cap f^{-1}(D_\delta^*) \to D_\delta^*$$ 
given by $\psi(x)=f(x)$ is a locally trivial fibration. 

The map $\psi$ depends, of course, on our choices of $\epsilon$ and
$\delta$. However, if $\epsilon',
\delta'$ is another pair of positive numbers satisfying the above conditions, the fibration associated
to these choices is fiber homotopy equivalent to the one above (see Definition 1.5 in Chapter 3,
\S1 of \cite{dimca1992singularities} for the definition of a fiber homotopy equivalence). We are thus
justified in calling $\psi$ the \emph{Milnor fibration} associated to $f$.

\begin{rem}
The Milnor fibration was originally introduced in \cite{milnor1968singular}. The above construction is not the same as the construction of the
Milnor fibration in \cite{milnor1968singular} and is due to L\^e (\cite{le1976some}). The two
constructions yield fiber homotopy equivalent fibrations (\cite{dimca1992singularities} Chapter 3, \S 1).

\end{rem}

Choose $t \in D_\delta^*$. The fiber of $\psi$ over $t$ is called the \emph{Milnor fiber of $f$ over $t$}; we will denote
it by $F_f$. $F_f$ is independent of our choices of $\epsilon$, $\delta$, and $t$ up to homotopy equivalence, so we suppress these choices in the
notation, and we will often refer to $F_f$ as just the Milnor fiber of $f$. However, these choices will be significant at several
points later on.

If $\C[x_1, \dots, x_n]_{(x_1, \dots,
x_n)}/(f)$ is $\IHS$ (see Definition~\ref{IHS}), set
$$\mu:=\dim_\C\frac{\C[[x_1, \dots, x_n]]}{(\frac{\partial f}{\partial x_1},
\dots, \frac{\partial f}{\partial x_n})} < \infty.$$

\begin{thm}[\cite{milnor1968singular} Theorem 6.5]
\label{bouquet}
If $\C[x_1, \dots, x_n]_{(x_1, \dots,
x_n)}/(f)$ is $\IHS$, $F_f$ is homotopy equivalent to a wedge sum of $\mu$ copies of $S^{n-1}$.
\end{thm}

\begin{rem} Since $\psi$ restricts to a fibration over a circle, $F_f$ comes equipped with a monodromy homeomorphism
$$h: F_f \xrightarrow{\cong} F_f.$$
\end{rem}

\subsubsection{The Sebastiani-Thom homotopy equivalence}
\label{TStopp}
We recall the definition of the join of two topological spaces:

\begin{defn}
Let $X$ and $Y$ be compact Hausdorff spaces. The \emph{join}
of $X$ and $Y$, denoted $X * Y$, is the quotient of $X \times Y \times I$ by the
relations
$$(x_1, y, 0) \sim (x_2, y, 0)$$
$$(x, y_1, 1) \sim (x, y_2, 1)$$
equipped with the quotient topology.
\end{defn}

\begin{rem}
\label{joincone}
We express the cone $CX$ over a
compact Hausdorff space $X$ explicitly as the quotient of
$$X\times [0,1]$$
by the relation $$(x_1, 0) \sim (x_2, 0)$$ for all $x_1, x_2 \in X$. When $X$ and $Y$ are
compact Hausdorff, 
$X * Y$ is homeomorphic to $(CX \times Y) \cup (X
\times CY) \subseteq CX \times CY;$
here, we identify $X$ and $Y$ with the subsets $X \times \{1\}$ and
$Y \times \{1\} $ of $CX$ and $CY$, respectively. By \cite{brown2006topology} 5.7.4, an explicit homeomorphism 
$$CX \times CY \xrightarrow{\cong} C(X*Y)$$
is given by
$$(x, t, y, t') \mapsto ((x, y, \frac{t}{2t'}), t') \text{, if } t' \ge
t \text{, } t' \ne 0$$
$$(x,t,y,t') \mapsto ((x, y, 1-\frac{t'}{2t}), t) \text{, if } t \ge 
t' \text{, } t \ne 0$$
$$(x,0,y,0) \mapsto ((x,y,0),0),$$
and this map restricts to a homeomorphism
$$w:(CX \times Y) \cup (X
\times CY) \xrightarrow{\cong} X * Y .$$
\end{rem}

Now, suppose $f \in \C[x_1, \dots, x_n]$, $f'
\in \C[y_1, \dots, y_m]$, and $f(0) = 0 = f'(0)$. Assume \newline
$R:=\C[x_1, \dots, x_n]_{(x_1, \dots, x_n)}/(f)$ and $R':=\C[y_1, \dots,
y_m]_{(y_1, \dots, y_m)}/(f')$ are
$\IHS$ (see Definition~\ref{IHS}). Let $f \oplus f'$ denote the sum of $f$ and $f'$ thought of as
an element of $\C[x_1, \dots, x_n, y_1, \dots, y_m]$. The
following theorem of Sebastiani-Thom relates the Milnor fibers of
$f$, $f'$, and $f \oplus f'$:

\begin{thm}[\cite{sebastiani1971resultat}]
\label{TStop}
There is a homotopy equivalence
$$\ST: F_f * F_{f'} \to F_{f \oplus f'}$$
that is compatible with monodromy; that is, the square 
\[
\begin{CD}
F_f * F_{f'} @>{\ST}>> F_{f \oplus f'}\\
@V{h * h}VV  @V{h}VV \\
F_f * F_{f'}  @>{\ST}>> F_{f \oplus f'}\\
\end{CD}
\]
commutes up to homotopy.
\end{thm}

\begin{rem}
By results of Oka in \cite{oka1973homotopy}, the assumption in Theorem~\ref{TStop} that
$R$ and $R'$ are $\IHS$ is not necessary if $f$ and
$f'$ are quasi-homogeneous.
\end{rem}

We refer the reader to Section 2.7 of \cite{arnold2012singularities} and $\S$3 of Chapter
3 in \cite{dimca1992singularities} for discussions related to
Theorem~\ref{TStop}. We now exhibit an explicit
map realizing the homotopy equivalence in Theorem~\ref{TStop}, following Section 2.7 of \cite{arnold2012singularities}.

Choose real numbers $\epsilon'', \delta'',$ such that the map
$$B_{\epsilon''} \cap (f \oplus f')^{-1}(D_{\delta''}^*) \to D_{\delta''}^*$$
given by $x \mapsto (f \oplus f')(x)$ is a locally trivial fibration, as above.

Similarly, choose $\epsilon, \delta$ and $\epsilon', \delta'$, as well
as $t'' \in D_{\delta''}^*$,
so that the analogous maps
$$B_{\epsilon} \cap f^{-1}(D_{\delta}^*) \to D_{\delta}^*$$
$$B_{\epsilon'} \cap (f')^{-1}(D_{\delta'}^*) \to D_{\delta'}^*$$
are locally trivial fibrations, and also so that
\begin{itemize}
\item[(a)] $\epsilon, \epsilon'$ are sufficiently small so that
$B_\epsilon \times B_{\epsilon'} \subseteq B_{\epsilon''}$.
\item[(b)] $|t''| < \min\{\delta, \delta'\}$.
\end{itemize}

Set $F_f$, $F_{f'}$, and $F_{f \oplus f'}$ to be the Milnor fibers of $f$,
$f'$, and $f \oplus f'$ over $t''$. Applying Lemma 2.10 in
\cite{arnold2012singularities}, choose a continuous map
$$H:  CF_f \to B_\epsilon$$ 
such that

\begin{itemize}
\item $H(x, 1) = x \in F_f \subseteq B_\epsilon$,
\item $H(-, s): F_f \to B_\epsilon$ maps into the Milnor fiber $B_\epsilon
  \cap f^{-1}(st'')$ for $s \in (0, 1)$, and 
\item $H(x, 0) = 0$ for all $x \in F_f$
\end{itemize}

\begin{exm}
\label{qhnotihs}
If $f$ is quasi-homogeneous of degree $d$ with weights $w_1, \dots,
w_d$, such a map $H$ may be given by
$$(x, s) \mapsto ({s}^{\frac{w_1}{d}}x_1, \dots,
{s}^{\frac{w_n}{d}}x_n).$$
Notice our assumption that $R$ is $\IHS$ is not needed here.
\end{exm}

Choose $H'$ similarly for the Milnor fiber $F_{f'}$. By the discussion on pages 54-55 of
\cite{arnold2012singularities} and Remark~\ref{joincone},
there is a homotopy equivalence
$$g: CF_f
\times F_{f'} \cup F_f \times CF_{f'} \to F_{f \oplus
  f'}$$
given by 
$$(x,s,y,s') \mapsto (H(x, \frac{1+s - s'}{2}), H'(y,
\frac{1-s+s'}{2})).$$

Composing, one has a homotopy equivalence 
$$ g \circ w^{-1}: F_f * F_{f'} \to F_{f \oplus f'},$$
where $w$ is the homeomorphism in Remark~\ref{joincone}. The homotopy
equivalence $g \circ w^{-1}$ enjoys the same properties as the map $\ST$
in Theorem~\ref{TS}.

\begin{rem} 
\label{TScone} 
$g$ extends to a homotopy equivalence
of pairs
$$G: ( CF_f \times CF_{f'}, CF_f
\times F_{f'} \cup F_f \times CF_{f'}) \to (B_{\epsilon''}, F_{f \oplus
  f'})$$
that maps a point $(x,s,y,s')$ to 
$$(H(x, \frac{s}{2}), H'(y, \frac{2s'-s}{2}) \text{, if } s \le s'
\text{, } s' \ne 0$$
$$(H(x, \frac{2s-s'}{2}),H'(y, \frac{s'}{2}) \text{, if } s' \le s
\text{, } s \ne 0$$
$$0 \text{, if } s=0=s'.$$

\end{rem}

\begin{rem}
\label{qhts}
When $f$ and $f'$ are quasi-homogeneous (and $R, R'$ are not necessarily $\IHS$), we may use
Example~\ref{qhnotihs} to build a homotopy equivalence $g: CF_f
\times F_{f'} \cup F_f \times CF_{f'} \to F_{f \oplus
  f'}$ in the same way as above (\cite{dimca1992singularities} Chapter
3, Remark 3.19').
\end{rem}


\subsubsection{An analogue of the Milnor fibration for polynomials over $\R$}
\label{realmilnorfiber}
Now, suppose $f \in \R[x_1, \dots, x_n]$ and $f(0) = 0$. One may construct a locally
trivial fibration
$$\psi: B_\epsilon \cap f^{-1}((-\delta, 0) \cup
(0, \delta)) \to
(-\delta, 0) \cup
(0, \delta)$$
for some $\epsilon > 0$ and $\delta$ such that $0 < \delta << \epsilon$ in the same way as above, where
$B_\epsilon$ is now the closed ball of radius
$\epsilon$ centered at the origin in $\R^{n}$. 

But now, fibers over $(-\delta, 0)$ and $(0, \delta)$ need not be
homotopy equivalent. For instance, if $f = x_1^2 + \dots + x_n^2$, the positive
fibers of $\psi$ are homeomorphic to $S^{n-1}$, while the
negative fibers are empty.

Choose $t \in (0, \delta)$ and $t' \in (-\delta, 0)$. The fiber of $\psi$ over $t$ is called the \emph{positive Milnor fiber of $f$ over $t$}, denoted by $F_f^+$, and the fiber of $\psi$ over $t'$ is called the \emph{negative Milnor fiber of $f$ over $t'$}, denoted $F_f^-$. As in the complex case, $F_f^+$ and $F_f^-$ are independent of our choices of $\epsilon$, $\delta$, $t$, and $t'$ up to homotopy equivalence, so we suppress these choices in our
notation, and we will often refer to $F_f^+$ and $F_f^-$ as just the positive and negative Milnor fibers of $f$.

The topology of the real Milnor fibers is more
complicated than that of the complex Milnor fiber. However, there is a version of Theorem~\ref{TStop}
for real Milnor fibers of quasi-homogeneous polynomials. Suppose 
$f \in \R[x_1, \dots, x_n] \text{, } f' \in \R[y_1, \dots, y_m]$
are quasi-homogeneous and nonconstant. If $F^+_f$ and $F^+_{f'}$ are nonempty, there is a homotopy equivalence
$$F^+_f * F^+_{f'} \to F^+_{f \oplus f'}$$
(\cite{dimca1992real} Remark 11).
Moreover, the homotopy equivalence may be constructed as in
Remark~\ref{qhts}; that is, one has a homotopy equivalence
of pairs
$$G: (CF^+_f \times C F^+_{f'},CF^+_f
\times F^+_{f'} \cup F^+_f \times CF^+_{f'}) \to (B_{\epsilon''},F^+_{f \oplus f'}).$$
Since $F_f^-=F_{-f}^+$, one has a similar result for negative Milnor
fibers.


\subsection{Relative topological $K$-theory}
\label{$K$-theory}
We introduce some facts concerning relative topological
$K$-theory. All of the results in this section are essentially due to
Atiyah-Bott-Shapiro in \cite{atiyah1964clifford}, but we modify their
exposition at several points to suit our purposes. 

Let $X$ be a compact topological space, and let $Y$ be a closed
subspace of $X$ such that there exists a homotopy
equivalence of pairs between $(X,Y)$ and a finite CW pair; we construct a category $\fC_1(X,Y)$ from $(X,Y)$ in the following way:

$\bullet$ An object of $\fC_1(X,Y)$ is a pair of real vector bundles $V_1$,
$V_0$ over $X$ equipped with an isomorphism
$$V_1|_Y  \xrightarrow{\sigma} V_0|_Y.$$
Denote objects of $\fC_1(X,Y)$ by $(V_1, V_0; \sigma)$.

$\bullet$ Morphisms in $\fC_1(X,Y)$ are pairs of morphisms of vector
bundles over X
$$\alpha_1: V_1 \to V_1' \text{, } \alpha_0: V_0 \to V_0'$$
such that the following diagram of maps of vector bundles over $Y$ commutes:
\[
\begin{CD}
V_1|_Y @>{\sigma}>> V_0|_Y\\
@V{\alpha_1|_Y}VV  @V{\alpha_0|_Y}VV \\
V_1'|_Y @>{\sigma'}>>V_0'|_Y\\
\end{CD}
\]

We write morphisms in $\fC_1(X,Y)$ as ordered pairs $(\alpha_1, \alpha_0)$.
\begin{rem}
The reason for the subscript in the notation $\fC_1(X,Y)$ is that, for
any $n \ge 1$, one
may similarly build a category $\fC_n(X,Y)$ with objects given by ordered $(n+1)$-tuples of
vector spaces on $X$ whose restrictions to $Y$ fit into an exact
sequence (cf. \cite{atiyah1964clifford} \S7).
\end{rem}

\begin{rem}
We will work with real vector bundles throughout this section;
however, there is an analogous version of every result in this section
for complex vector bundles.
\end{rem}

The following facts about $\fC_1(X,Y)$ are easily verified:
\begin{itemize}

\item If $(V_1, V_0; \sigma)$ and $(V_1', V_0'; \sigma')$ are objects
in $\fC_1(X,Y)$, $(V_1 \oplus V_1', V_0 \oplus V_0', \sigma \oplus \sigma')$ is
their coproduct. 

\item $\fC_1(X,Y)$ is an additive category.

\item A map $g: (X_1, Y_1) \to (X_2, Y_2)$ of pairs of spaces as above induces a functor
$$g^*: \fC_1(X_2, Y_2) \to \fC_1(X_1, Y_1)$$
via pullback. 

\item A morphism $(\alpha_1, \alpha_0)$ in $\fC_1(X,Y)$ 
is an isomorphism (resp. monomorphism,
epimorphism) if and only if $\alpha_1$ and $\alpha_0$ are isomorphisms
(resp. monomorphisms, epimorphisms) of vector bundles over $X$.

\end{itemize}

We shall call an object of $\fC_1(X,Y)$ \emph{elementary} if it is
isomorphic to an object of the form $(V, V; \id_{V|_Y})$. It is easy
to check that $(V_1, V_0; \sigma)$ is elementary if and only if $\sigma$ can be extended to an isomorphism $\widetilde{\sigma}: V_1
\to V_0$.

If $V$ and $V'$ are objects in $\fC_1(X,Y)$, we will say $V \sim V'$ if and
only if there exist elementary objects $E, E'$ such that 
$$V \oplus E \cong V' \oplus E'.$$ 

The relation $\sim$ is an equivalence relation. Let $L_1(X,Y)$ denote
the commutative monoid of equivalence classes under $\sim$ with
operation $\oplus$. We shall denote by $[V_1, V_0; \sigma]$ the class in
$L_1(X,Y)$ represented by $(V_1, V_0, \sigma)$.

\begin{rem}
\label{L1pullback}
Let $ (X_1, Y_1)$, $(X_2, Y_2)$ be pairs of spaces as above, and let $g: (X_1,
Y_1) \to (X_2, Y_2)$ be a map of pairs. Then the functor
$$g^*: \fC_1(X_2, Y_2) \to \fC_1(X_1, Y_1)$$
applied to an elementary object is again elementary. Hence, $g^*$
induces a map of monoids
$$L_1(X_2, Y_2) \to L_1(X_1, Y_1).$$
\end{rem}

The main reason we are interested in the monoid $L_1(X,Y)$ is the
following result:

\begin{prop}[Atiyah-Bott-Shapiro, \cite{atiyah1964clifford}]
\label{bigtheoremABS}
There exists a unique natural homomorphism
$$\chi: L_1(X,Y) \to KO^0(X,Y)$$
which, when $Y = \emptyset$, is given by
$$\chi(E) = [V_0] - [V_1].$$

Moreover, $\chi$ is an isomorphism.
\end{prop}

In particular, $L_1(X,Y)$ is an abelian group.

Let $(X, Y)$, $(X', Y')$ be pairs as above. We conclude this section
by exhibiting a product map
$$L_1(X, Y) \otimes L_1(X', Y') \to L_1(X \times X', X
\times Y' \cup Y \times X')$$
that agrees, via $\chi$, with the usual product on relative $K$-theory. 

Let $V=(V_1, V_0; \sigma) \in \Ob(\fC_1(X, Y))$ and $V'=(V_1', V_0';
\sigma') \in
\Ob(\fC_1(X', Y'))$. By Proposition 10.1 in \cite{atiyah1964clifford}, we may lift
$\sigma, \sigma'$ to maps
$\widetilde{\sigma}, \widetilde{\sigma}'$ of bundles over $X$ and
$X'$, respectively.

Thinking of 
$$0 \to V_1 \xrightarrow{\widetilde{\sigma}} V_0 \to 0$$
$$0 \to V_1' \xrightarrow{\widetilde{\sigma}'} V_0' \to 0$$
as complexes of bundles with $V_1, V_1'$ in degree $1$ and $V_0, V_0'$ in degree
$0$, we may take their tensor product 
$$0 \to V_1 \otimes V_1' \xrightarrow{\tau_2}  (V_1 \otimes
V_0') \oplus (V_0 \otimes V_1') \xrightarrow{\tau_1} V_0 \otimes V_0'
\to 0,$$
where
$$\tau_1=\begin{pmatrix} \widetilde{\sigma} \otimes \id_{V_0'} &
  \id_{V_0} \otimes  \widetilde{\sigma}'  \end{pmatrix}$$
$$\tau_2=\begin{pmatrix} -\id_{V_1} \otimes \widetilde{\sigma}' \\
  \widetilde{\sigma} \otimes \id_{V_1'} \end{pmatrix}$$

The result is a complex of vector bundles over $X \times X'$ that is
exact upon restriction to $X
\times Y' \cup Y \times X'$.

Choose a splitting $\pi$ of $\tau_2|_{ X
\times Y' \cup Y \times X'}$. Then,
$$[(V_1 \otimes V_0') \oplus (V_0 \otimes V_1'), (V_0 \otimes V_0')
\oplus (V_1 \otimes V_1'); \begin{pmatrix} \tau_1|_{ X
\times Y' \cup Y \times X'} \\ \pi
  \\ \end{pmatrix}]$$
is an element of $L_1( X \times X', X
\times Y' \cup Y \times X')$.

One may define monoids $L_n(X,Y)$ involving longer sequences
of bundles; see \cite{atiyah1964clifford} Definition 7.1 for details. Denote
elements of $L_n(X,Y)$ by 
$$[V_n, \dots, V_0; \sigma_n, \dots,
\sigma_1].$$ 
There is a map
$$j_n: L_1(X,Y) \to L_n(X,Y)$$
given by
$$[V_1, V_0; \sigma] \mapsto [0, \dots, 0, V_1, V_0; 0, \dots, 0,
\sigma],$$
and, by Proposition 7.4 in \cite{atiyah1964clifford}, $j_n$ is an
isomorphism for all $n$.

We will need the following technical lemma:

\begin{lem}
\label{splitting}
Let $(X,Y)$ be a pair as above, and let $[V_2, V_1, V_0; \sigma_2, \sigma_1] \in
L_2(X,Y)$. If $\pi$ is a splitting of $\sigma_2$, 
$$j_2([V_1, V_0 \oplus V_2; \begin{pmatrix} \sigma_1 \\ \pi
  \\ \end{pmatrix}]) = [V_2, V_1, V_0; \sigma_2, \sigma_1].$$
\end{lem}

\begin{proof}
First, suppose $\dim(V_1) > \dim(V_2) + \dim(X)$. Apply Lemma 7.2 in
\cite{atiyah1964clifford} to construct a monomorphism
$$h: V_2 \to V_1$$
that extends $\sigma_2$.
By the proof of Lemma 7.3 in \cite{atiyah1964clifford}, 
$$j_2([\coker(h), V_0; \overline{\sigma_1}]) = [V_2, V_1, V_0; \sigma_2, \sigma_1],$$
and so

$$j_2([\coker(h) \oplus V_2, V_0 \oplus V_2; A]) =
[V_2, V_1, V_0; \sigma_2, \sigma_1],$$
where 
$$A=\begin{pmatrix}
  \overline{\sigma_1} & 0 \\ 0 & \id_{V_2}|_Y \\ \end{pmatrix}.$$
Hence, it suffices to show
$$[\coker(h) \oplus V_2, V_0 \oplus V_2; A]=[V_1, V_0 \oplus V_2; \begin{pmatrix} \sigma_1 \\ \pi
  \\ \end{pmatrix}]$$

Choose a splitting $s$ of $h$, and let 
$$p: V_1 \to \coker(h)$$ 
denote the canonical map. Then we have an isomorphism
$$\begin{pmatrix} p \\ s \\ \end{pmatrix}: V_1 \to \coker(h) \oplus
V_2.$$

Since $s|_Y$ is a splitting of $\sigma_2$, we also have an isomorphism
$$\begin{pmatrix} \sigma_1 \\ s|_Y \\ \end{pmatrix}: V_1|_Y \to V_0|_Y \oplus
V_2|_Y.$$

We have a commutative square

$$
\xymatrix{
V_1|_Y \ar[r]^-{\begin{pmatrix}
    \sigma_1 \\ s|_Y \\ \end{pmatrix}} \ar[dd]^-{\begin{pmatrix} p|_Y \\ s|_Y \\ \end{pmatrix}} & V_0|_Y \oplus
V_2|_Y  \ar[dd]^-{\id_{V_0|_Y \oplus
V_2|_Y}}\\
&\\
\coker(h)|_Y \oplus V_2|_Y \ar[r]^-{A} & V_0|_Y \oplus
V_2|_Y
}
$$

Thus,

$$[\coker(h) \oplus V_2, V_0 \oplus V_2; A]=[V_1, V_0 \oplus V_2; \begin{pmatrix} \sigma_1 \\ s|_Y
  \\ \end{pmatrix}].$$

Notice that we have an object
$$[V_1 \times I, (V_0 \oplus V_2) \times I; t \begin{pmatrix} \sigma_1 \\ s|_Y
  \\ \end{pmatrix} + (1-t) \begin{pmatrix} \sigma_1 \\ \pi
  \\ \end{pmatrix}]$$
in $\fC_1(X \times I ,Y \times I)$ whose restrictions to $X \times \{0\}$
and $X \times \{1\}$ are $[V_1, V_0 \oplus V_2; \begin{pmatrix} \sigma_1 \\ \pi
  \\ \end{pmatrix}]$ and $[V_1, V_0 \oplus V_2; \begin{pmatrix} \sigma_1 \\ s|_Y
  \\ \end{pmatrix}]$, respectively. It now follows from Proposition
9.2 in \cite{atiyah1964clifford} that 
$$[V_1, V_0 \oplus V_2; \begin{pmatrix} \sigma_1 \\ s|_Y
  \\ \end{pmatrix}]=[V_1, V_0 \oplus V_2; \begin{pmatrix} \sigma_1 \\ \pi
  \\ \end{pmatrix}] .$$

This finishes the case where $\dim(V_1) > \dim(V_2) + \dim(X)$.

For the general case, choose a bundle $E$ such that 
$$\dim(E) + \dim(V_1) > \dim(V_2) + \dim(X).$$

Define
$$U := [V_2, V_1 \oplus E, V_0 \oplus E; \begin{pmatrix} \sigma_2 \\
  0 \end{pmatrix}, \begin{pmatrix} \sigma_1 && 0 \\
  0 && \id_{E}|_Y \end{pmatrix}],$$
$$U' := [V_1 \oplus E, V_0 \oplus E \oplus V_2 ; \begin{pmatrix}
  \sigma_1 && 0 \\0 && \id_{E}|_Y \\ \pi &&  0 \\ \end{pmatrix}]$$

Notice that 
$$[V_2, V_1, V_0; \sigma_2, \sigma_1]=U,$$
and 
$$[V_1, V_0 \oplus V_2; \begin{pmatrix} \sigma_1 \\ \pi
  \\ \end{pmatrix}] = U',$$
so that it suffices to show that $j(U') = U$. Since
$\begin{pmatrix} \pi & 0  \end{pmatrix}$ is a splitting of  $\begin{pmatrix}
  \sigma_2  \\0 \\  \end{pmatrix}$, this follows from the case we have already considered.
\end{proof}

Now, the pairing
$$L_1(X, Y) \otimes L_1(X', Y') \to L_1(X \times X', X
\times Y' \cup Y \times X')$$ described in Proposition 10.4 of \cite{atiyah1964clifford} is given by
sending a simple tensor
$$[V_1, V_0; \sigma] \otimes [V_1', V_0'; \sigma']$$
to
$$j_2^{-1}  (  [V_1 \otimes V_1' , (V_1 \otimes
V_0') \oplus (V_0 \otimes V_1') , V_0 \otimes V_0'; \tau_2|_{X
\times Y' \cup Y \times X'},
\tau_1|_{X
\times Y' \cup Y \times X'}]);$$
this follows from the proof of Proposition 10.4.

Thus, by Lemma~\ref{splitting}, the map
$$\Ob(\fC_1(X, Y)) \times \Ob(\fC_1(X', Y')) \to L_1(X \times X', X \times
Y' \cup Y \times X')$$
given by
$$(V, V') \mapsto [(V_1 \otimes V_0') \oplus (V_0 \otimes V_1'), (V_0 \otimes V_0')
\oplus (V_1 \otimes V_1'); \begin{pmatrix} \tau_1|_{X
\times Y' \cup Y \times X'} \\ \pi
  \\ \end{pmatrix}]$$
determines
\begin{itemize}
\item[(a)] a well-defined pairing on $\Ob(\fC_1(X, Y)) \times \Ob(\fC_1(X', Y'))$ up to our choices of liftings $\widetilde{\sigma}$,
$\widetilde{\sigma'}$ and splitting $\pi$, and 
\item[(b)] a pairing 
$$L_1(X, Y) \otimes L_1(X', Y') \to L_1(X \times X', X
\times Y' \cup Y \times X')$$
that coincides with the pairing in Proposition 10.4 of
\cite{atiyah1964clifford}.
\end{itemize}

Let $[V]$,$[V']$ denote the classes represented by $V$ and $V'$ in $L_1(X,
Y)$ and $L_1(X', Y')$. Define $$[V] \otimes_{L_1} [V']:=[(V_1 \otimes V_0') \oplus (V_0 \otimes V_1'), (V_0 \otimes V_0')
\oplus (V_1 \otimes V_1'); \begin{pmatrix} \tau_1|_{X
\times Y' \cup Y \times X'} \\ \pi
  \\ \end{pmatrix}].$$

\begin{rem}
\label{10.4}
By Proposition 10.4 in \cite{atiyah1964clifford} and the above remarks,
$$\chi([V]) \otimes \chi([V']) = \chi([V] \otimes_{L_1} [V']).$$
\end{rem}


\subsection{A generalized Atiyah-Bott-Shapiro construction applied to
  matrix factorizations}
\label{ABS}

In this section, we construct the maps $\phi^\C_f$ and $\phi^\R_f$
described in the introduction. We begin with a discussion of the
Atiyah-Bott-Shapiro construction (\cite{atiyah1964clifford} Part
III). Following Atiyah-Bott-Shapiro,
we work with real Clifford algebras and $KO$-theory, and we
point out that one may perform a similar construction involving complex
Clifford algebras and $KU$-theory.


\subsubsection{The Atiyah-Bott-Shapiro construction}

Define
$$q_n:=-x_1^2 - \cdots - x_n^2 \in \R[x_1, \dots, x_n]$$
for all $n \ge 1$, and set $C_n:=\cliff_\R(q_n)$. We also set
$C_0:=\R$; we will think of $C_0$ as a $\Z/2\Z$-graded algebra concentrated in
degree 0.

Let $M(C_n)$ denote the free abelian group generated by isomorphism
classes of finitely-generated, indecomposable $\Z/2\Z$-graded left $C_n$-modules. There are evident injective maps
$$i_n: C_n \to C_{n+1}$$
for all $n \ge 0$; these injections induce homomorphisms
$$i_n^*: M(C_{n+1}) \to M(C_n)$$
via restriction of scalars. Set
$$A_n:= M(C_n)/i_n^*(M(C_{n+1})).$$

Define $D^n$ to be the closed disk of radius 1 in $\R^n$. An important special case of the classical Atiyah-Bott-Shapiro construction is
the group isomorphism
$$\alpha_n: A_n \xrightarrow{\cong} L_1(D^n, \partial D^n)$$
that appears in \cite{atiyah1964clifford} Theorem 11.5. $\alpha_n$ is defined as follows: let $M=M_1 \oplus M_0$ be a finitely generated $\Z/2\Z$-graded left
$C_n$-module. We use the $\R$-vector spaces $M_1$ and $M_0$ to construct real vector bundles over $D^n$:
$$V_1:= D^n \times M_1$$
$$V_0:=D^n \times M_0$$
and we define a map
$$\sigma: V_1 \to V_0$$
given by $(x, m) \mapsto (x, x \cdot m)$, where $\cdot$ denotes the
action of $C_n$ on $M$. Here, we are thinking of $D^n \subseteq \R^n$
as a subset of $C_n$. Notice that $\sigma$ restricts to an
isomorphism of bundles over $\partial D^n$. Thus, we have constructed
an element
$[V_1, V_0; \sigma] \in L_1(D^n, \partial D^n).$ Define
$$\alpha_n([M]) = [V_1, V_0; \sigma].$$
We refer the reader to \cite{atiyah1964clifford} for verification that
the mapping 
$$[M] \mapsto [V_1, V_0; \sigma]$$
is well-defined on the quotient $A_n$ and determines an isomorphism. 


\subsubsection{A more general construction}
\label{ABSsection}
Let $f \in (x_1, \dots, x_{n}) \subseteq Q:=\R[x_1, \dots,
x_n]$. Choose real numbers $\epsilon, \delta,$ and $t$ such that $\epsilon >0$, $0 <
\delta << \epsilon$, and $t \in (-\delta, 0)$ in such a way that we may construct
a negative Milnor fiber
$F_f^-$ as in
Section~\ref{realmilnorfiber}.

Denote by
$B_\epsilon$ the closed ball of radius $\epsilon$ in $\R^n$ centered
at the origin. We now construct
a map 
$$\Ob(\MF(Q,f)) \to L_1(B_\epsilon, F_f^-)$$
that
\begin{itemize}
\item[(a)] recovers the Atiyah-Bott-Shapiro construction via
  the Buchweitz-Eisenbud-Herzog equivalence (Theorem~\ref{BEH}) when $f=q_n$, and
\item[(b)] descends to a group homomorphism
$$K_0[\MF(Q, f)] \to L_1(B_\epsilon, F_f^-).$$
\end{itemize}

We emphasize that a similar construction involving complex polynomials
and their Milnor fibers may be performed \emph{mutatis
mutandis}. One may also perform the following construction using the
positive Milnor fiber $F_f^+$ of $f$.

Let $P=(P_1 \darrow{d_1}{d_0} P_0)$ be a matrix factorization of $f$ over $Q$. Denote by $C(B_\epsilon)$ the ring of
$\R$-valued continuous functions on $B_\epsilon$. Applying extension of scalars along the inclusion
$$Q \into C(B_\epsilon),$$
we obtain a map
$$P_1 \otimes_Q C(B_\epsilon) \xrightarrow{d_1 \otimes \id} P_0
\otimes_Q C(B_\epsilon)$$
of finitely generated projective $C(B_\epsilon)$-modules.

The category of real vector bundles over $B_\epsilon$ is
equivalent to the category of finitely generated projective
$C(B_\epsilon)$-modules; on objects, the equivalence sends a bundle to
its space of sections.
Let 
$$V_1 \xrightarrow{d_1} V_0$$
be a map of real vector bundles over $B_\epsilon$ corresponding to the above map $d_1 \otimes \id$ under this
equivalence. Since $d_1 \circ d_0 = f \cdot \id_{P_0}$ and $d_0 \circ d_1 = f \cdot \id_{P_1}$, and since the restriction of the polynomial $f$, thought of as a map $\R^n \to \R$, to $F_f^- = B_{\epsilon} \cap f^{-1}(t)$ is constant with value $t \ne 0$, $d_1|_{F_f^-}$ is an isomorphism of vector bundles on $F_f^-$. Its inverse is the
restriction to $F_f^-$ of the map $V_0 \to V_1$ determined by 
$$P_0
\otimes_Q C(B_\epsilon) \xrightarrow{\frac{1}{t} (d_0 \otimes \id)}
P_1 \otimes_Q C(B_\epsilon).$$
Define $\Phi^\R_f(P_1 \darrow{d_1}{d_0} P_0) = (V_1, V_0; d_1|_{F_f^-}) \in \Ob
(\fC_1(B_\epsilon, F_f^-))$. 

\begin{rem}
\label{BVS}
The map analogous to $\Phi^\R_f$ in the setting of polynomials over $\C$
and $KU$-theory appears in
\cite{buchweitz2012index}; we discuss this in detail in Section~\ref{thetavanishing}.
\end{rem}

A morphism in $Z^0\MF(Q,f)$ determines a
morphism in $\fC_1(B_\epsilon, F_f^-)$ in an obvious way (see
Section~\ref{mf} for the definition of the category $Z^0\MF(Q,f)$). Hence, we have shown:

\begin{prop}
There is an additive functor 
$$\Phi^\R_f: Z^0\MF(Q,f) \to \fC_1(B_\epsilon, F_f^-)$$
given, on objects, by 
$$(P_1 \darrow{d_1}{d_0} P_0) \mapsto [V_1, V_0; d_1|_{F_f^-}].$$
\end{prop}

In particular, we have a map
$$\Ob(\MF(Q,f)) \to L_1(B_\epsilon, F_f^-).$$

Suppose $f = q_n$. Then $\epsilon$ can be
chosen to be 1 in the construction of the negative Milnor fiber
$F_f^-$, and the fiber can be
chosen to be exactly $S^{n-1} \subseteq \R^n$.

Let $\Iso([\MF(Q,f)])$ and $\Iso(\fgmod_{\Z/2\Z}(\cliff_\R(q_n)))$
denote the sets of isomorphism classes of objects in $[\MF(Q,f)]$ and $\fgmod_{\Z/2\Z}(\cliff_\R(q_n))$. It is easy to
check that one has a commutative triangle
$$
\xymatrix{
\Iso(\fgmod_{\Z/2\Z}(\cliff_\R(q_n))) \ar[r]^-{\ABS}  \ar[dd]^-{{[\Theta]}} & L_1(B_1, F_f^-) \\
& \\
\Iso({[\MF(Q,f)]}) \ar[ruu]^-{\Phi^\R_f}
}
$$
where $[\Theta]$ denotes the bijection on isomorphism classes
induced by the explicit construction $\Theta$ of the
Buchweitz-Eisenbud-Herzog equivalence (Theorem~\ref{BEH}) provided in the proof of
Theorem 14.7 of \cite{yoshino1990maximal}, and $\ABS$ denotes the Atiyah-Bott-Shapiro construction. Hence, our construction recovers the
Atiyah-Bott-Shapiro construction via the Buchweitz-Eisenbud-Herzog
equivalence when $f=q_n$.

Our next goal is to show that $\Phi^\R_f$ induces a map on $K$-theory:

\begin{prop}
\label{ABSmap}
$\Phi^\R_f$ induces a group homomorphism
$$\phi^\R_f: K_0[\MF(Q, f)] \to L_1(B_\epsilon, F_f^-).$$
\end{prop}

We will
adopt the following notational conventions for the purposes of the
proof of Proposition~\ref{ABSmap}:

\begin{itemize}

\item[(1)]  A pair $(\epsilon, t)$ is a \emph{good pair} if $\epsilon
  >0$, $t < 0$, and the map 
$$\psi: B_\epsilon \cap f^{-1}((-\delta, 0) \cup (0, \delta)) \to (-\delta, 0) \cup (0, \delta)$$
from Section~\ref{realmilnorfiber} is a locally trivial
fibration for some $\delta>0$ such that $$0 < |t| < \delta <<
\epsilon.$$

\item[(2)] If $(\epsilon, t)$ is a good pair, we denote the negative
  Milnor fiber $B_\epsilon \cap f^{-1}(t)$ by $F_t^-$.
\end{itemize}

We will need the following technical lemma:

\begin{lem}
\label{technical}
Let $(\epsilon_1, t_1), (\epsilon_2, t_2)$ be good pairs. Then there is an
isomorphism 
$$g: L_1(B_{\epsilon_1}, F_{t_1}^-) \xrightarrow{\cong} L_1(B_{\epsilon_2},
F_{t_2}^-)$$
yielding a commutative triangle
$$
\xymatrix{
\Ob(\MF(Q,f)) \ar[r]^-{\Phi^\R_f} \ar[d]^-{\Phi^\R_f} & L_1(B_{\epsilon_2},
F_{t_2}^-) \\
L_1(B_{\epsilon_1}, F_{t_1}^-) \ar[ru]^-{g} 
}
$$
\end{lem}

\begin{proof}
The case where $t_1=t_2$ is immediate, so we may assume $t_1 \ne t_2$. First, suppose $\epsilon_1 = \epsilon_2$. Without loss, assume $t_2 < t_1$.

Set $F^-_{[t_2,t_1]}:=f^{-1}([t_2, t_1])$. Since the inclusions
$$F^-_{t_1} \hookrightarrow F^-_{[t_2,t_1]} \text{, }F^-_{t_2} \hookrightarrow F^-_{[t_2,t_1]}$$
are homotopy equivalences, the pullback maps
$$L_1(B_{\epsilon_1}, F_{[t_2,t_1]}) \to  L_1(B_{\epsilon_1},
F_{t_1}) \text{, } L_1(B_{\epsilon_1}, F_{[t_2,t_1]}) \to  L_1(B_{\epsilon_1}, F_{t_2})$$
are isomorphisms.

We have commuting triangles 
$$
\xymatrix{
\Ob(\MF(Q,f)) \ar[r]^-{\Phi^\R_f} \ar[d]^-{\Phi^\R_f} & L_1(B_{\epsilon_1},
F_{t_1}^-) \\
L_1(B_{\epsilon_1}, F_{{[t_2, t_1]}}^-) \ar[ru]^-{\cong} 
}
$$

for $i=1,2$. It follows that the result holds when $\epsilon_1 = \epsilon_2$.

For the general case, assume,
without loss, that $|t_2|
< |t_1|$. Then $(\epsilon_1, t_2)$ is also a good pair. By the cases we've
already considered, the result holds for the pairs $(\epsilon_1, t_1)$
and $(\epsilon_1, t_2)$, and also for the pairs $(\epsilon_1, t_2)$ and
$(\epsilon_2, t_2)$. Hence, the result holds for the pairs $(\epsilon_1,
t_1)$, $(\epsilon_2, t_2)$.
\end{proof}

We now prove Proposition~\ref{ABSmap}:

\begin{proof}
It is not hard to see that $\Phi^\R_f(P \oplus P') = \Phi^\R_f(P) \oplus
\Phi^\R_f(P')$; we need only show that $\phi^\R_f$ is well-defined. First, suppose $P \cong 0$ in $[\MF(Q,
f)]$. Then $\id_P$ is a boundary in $\MF(Q, f)$, and so $\id_P$
factors through a trivial matrix factorization, by Proposition
~\ref{trivialmf}. 

Write $P=(P_1 \darrow{d_1}{d_0} P_0).$ Since $P$ is a summand of a trivial matrix factorization,
$\coker(d_1)$ is a projective $Q/(f)$ module. Choose $g \in Q$ such
that $g(0) \ne 0$ and $\coker(d_1)_g$ is free over $Q_g/(f)$, and choose $\epsilon' \in (0, \epsilon)$
such that $B_{\epsilon'} \cap g^{-1}(0) = \emptyset$. The inclusion $Q \hookrightarrow Q_g$ induces a functor $\MF(Q, f) \to
\MF(Q_g, f).$ Choose $t'$ such that $(\epsilon', t')$ is a good pair. Applying Lemma~\ref{technical}, we have a commutative diagram
$$
\xymatrix{
\Ob(\MF(Q,f)) \ar[r] \ar[dd]^-{\Phi^\R_f} \ar[rdd]^-{\Phi^\R_f} & \Ob(\MF(Q_g,f)) \ar[dd]^-{\Phi^\R_f}\\
& \\
L_1(B_{\epsilon}, F_t^-) \ar[r]^-{\cong} & L_1(B_{\epsilon'},
F_{t'}^-)
}
$$
It is easy to see that the $\phi^\R_f$ is well-defined when $f=0$, so
assume $f \ne 0$. Then $f$ is a non-zero-divisor in $Q$, so we may apply Proposition~\ref{cokerfree} to conclude that the image of $P$ in $\Ob(\MF(Q_g, f))$
maps to $0$ via $\Phi^\R_f$. Hence, the map $\Phi^\R_f: \Ob(\MF(Q, f))
\to L_1(B_{\epsilon}, F_t^-)$ sends $P$ to $0$, as
required.

We now show that, if $\alpha: P \to P'$ is a morphism in $Z^0\MF(Q,f)$,
$\Phi^\R_f(P) \oplus \Phi^\R_f(\cone(\alpha))$ and
$\Phi^\R_f(P')$ represent the same class in $L_1(B_{\epsilon},
F_t^-)$. We start by showing $\Phi^\R_f(P[1]) = -\Phi^\R_f(P)$ in
$L_1(B_{\epsilon}, F_t^-)$. Write $\Phi^\R_f(P)=(V_1, V_0; d_1|_{F_t^-})$,
so that $\Phi^\R_f(P[1])=(V_0, V_1; -d_0|_{F_t^-})$. Since $\cone(\id_P)$ is contractible, the class represented by
$$\Phi^\R_f(\cone(\id_P))=(V_0 \oplus V_1, V_1 \oplus
V_0; \begin{pmatrix} d_0|_{F_t^-} & \id \\ 0 & -d_1|_{F_t^-} \\\end{pmatrix})$$
in $L_1(B_{\epsilon}, F_t^-)$ is $0$. The object 
$$((V_0 \oplus V_1) \times I, (V_1 \oplus
V_0) \times I; \begin{pmatrix} d_0|_{F_t^-} & s \cdot \id \\ 0 & -d_1|_{F_t^-} \\\end{pmatrix})$$
of $\fC_1(B_{\epsilon} \times I, F_t^- \times I)$ restricts to
$\Phi^\R_f(\cone(\id_P))$ at $s = 1$ and $\Phi^\R_f((P \oplus P[1])[1])$ at
$s=0$. Since $(P \oplus P[1])[1] \cong P \oplus P[1]$, we may use
Proposition 9.2 in \cite{atiyah1964clifford} to conclude that $\Phi^\R_f(P[1]) =
-\Phi^\R_f(P)$ in $L_1(B_{\epsilon}, F_t^-)$.

Now, we have
$$\Phi^\R_f(\cone(\alpha)) = (V_0 \oplus V_1, V_1 \oplus
V_0; \begin{pmatrix} d_0|_{F_t^-} & \alpha_1 \\ 0 & -d_1'|_{F_t^-} \\\end{pmatrix}).$$
Using Proposition 9.2 in \cite{atiyah1964clifford} in the same manner as above, we may conclude that
$\Phi^\R_f(\cone(\alpha))$ and
$\Phi^\R_f(P') \oplus \Phi^\R_f(P[1])$ represent the same class in
$L_1(B_{\epsilon}, F_t^-)$. 

Finally, suppose $\alpha: P \xrightarrow{\cong} P'$ is an isomorphism in
$[\MF(Q,f)]$. Then $\cone(\alpha)$ is contractible, and so the results
we just established imply that $\Phi^\R_f(P) = \Phi^\R_f(P')$. Since every
distinguished triangle in $[\MF(Q,f)]$ is isomorphic to one of the
form 
$$P \xrightarrow{\alpha} P' \to \cone(\alpha) \to P[1],$$
and we have shown that $\Phi^\R_f$ preserves such triangles, we are done.
\end{proof}


\subsubsection{The kernel and image of $\phi^\C_f$}
\label{thetavanishing}
Let $Q:=\C[x_1, \dots, x_n]$, and set $\m:=(x_1, \dots, x_n) \subseteq Q$. Fix $f \in \m$, and define $R := Q/(f)$. Assume the hypersurface $R$ has an isolated singularity at the origin in the sense of Definition~\ref{IHS}. Choose $\epsilon,\delta >0$ so that the map
$$B_\epsilon \cap f^{-1}(D_\delta^*) \to D_\delta^*$$
given by $x \mapsto f(x)$
is a locally trivial fibration, as in Section~\ref{milnorfiber}; let
$F_f$ denote the Milnor fiber of $f$ over some value $t \in D_{\delta}^*$. We wish to examine the kernel
and image of the map 
$$\phi_f^\C: K_0[\MF(Q, f)] \to L_1(B_\epsilon,
F_f).$$

Recall that, by Theorem~\ref{bouquet}, $F_f$ is homotopy equivalent to a
wedge sum of $\mu$ copies of $S^{n-1}$, where $\mu$ is the Milnor
number of $f$. Thus, 
\begin{displaymath}
   L_1(B_\epsilon, F_f) \cong KU^0(B_\epsilon, F_f) \cong KU^{-1}(F_f) \cong
\bigoplus_\mu KU^{-1}(S^{n-1}) \cong \left\{
     \begin{array}{ll}
       \Z^\mu & \text{if } n \text{ is even}\\
       0 & \text{if } n \text{ is odd}\\
     \end{array}
   \right.
\end{displaymath}

In particular, when $n$ is odd, $\phi^\C_f=0$.

As we noted in Section~\ref{milnorfiber}, $F_f$ is equipped
with a monodromy homeomorphism
$$h: F_f \xrightarrow{\cong} F_f.$$
Let $S \subseteq D_\delta^*$ denote the circle of radius $|t|$ centered at the origin, and set $E:=B_\epsilon \cap f^{-1}(S)$. One has a long exact
sequence, the \emph{Wang exact sequence} (\cite{dimca1992singularities} page
74)
$$\cdots \to H^i(E) \xrightarrow{j^*} H^i(F_f) \xrightarrow{h^* - 1} 
H^i(F_f) \to H^{i+1}(E) \to \cdots$$
where $j:F_f \into E$ is the
inclusion. One also has an automorphism $T: L_1(B_\epsilon, F_f)
\xrightarrow{\cong} L_1(B_\epsilon, F_f)$ induced by $h$.

We have the following result regarding the image of $\phi_f^\C$:

\begin{prop}
\label{fixedbymonodromy}
$\phi^\C_f(K_0[\MF(Q,f)]) \subseteq \ker(T -1)$.
\end{prop}

\begin{proof}
The result is obvious when $n$ is odd, since $L_1(B_\epsilon,
F_f)=0$ in this case. Suppose $n$ is even. Notice that $\phi^\C_f(K_0[\MF(Q,f)]) \subseteq
l^*(L_1(B_\epsilon, E))$, where $l: (B_\epsilon, F_f) \into (B_\epsilon, E)$ is the inclusion of pairs. Thus, the result follows from the commutative diagram
$$
\xymatrix{
L_1(B_\epsilon,E) \ar[r]^-{l^*} \ar[d]^-{\cong} \otimes \Q & L_1(B_\epsilon,F_f) \otimes \Q \ar[d]^-{\cong} \ar[r]^-{T-1} & L_1(B_\epsilon,F_f) \otimes \Q \ar[d]^-{\cong} \\
KU^{-1}(E) \otimes \Q \ar[r]^-{j^*} \ar[d] & KU^{-1}(F_f) \otimes \Q \ar[d]^-{\cong} \ar[r]^-{h^*-1} & KU^{-1}(F_f) \otimes \Q \ar[d]^-{\cong} \\
H^{n-1}(E;\Q) \ar[r]^-{j^*} & H^{n-1}(F_f;\Q) \ar[r]^-{h^*-1} & H^{n-1}(F_f;\Q)
}
$$
and the Wang exact sequence. The bottom-most vertical arrows are Chern class maps; the bottom-middle and
bottom-right vertical maps
are isomorphisms because $F_f$ has nonzero odd cohomology only in degree $n-1$.
\end{proof}

The map $\Phi^\C_f: \Ob(\MF(Q,f)) \to L_1(B_\epsilon, F_f)$ is
used in \cite{buchweitz2012index} to study
the Hochster theta pairing. We recall the definition of this pairing:

\begin{defn}
\label{thetadefn}
The \emph{Hochster theta pairing}
$$\theta: K_0[\MF(Q_\m,f)] \times K_0[\MF(Q_\m,f)] \to \Z$$
sends a pair $([P_1 \darrow{d_1}{d_0} P_0], [P_1' \darrow{d_1'}{d_0'} P_0'])$
to
$$l(\Tor^{R_\m}_{2}(\coker(d_1), \coker(d_1'))) -
l(\Tor^{R_\m}_{1}(\coker(d_1), \coker(d_1'))),$$
where $l$ denotes length as an $R_\m$-module.
\end{defn}

\begin{rem} 
Our assumption that $R$ is $\IHS$ guarantees that the lengths in Definition~\ref{thetadefn} are finite. The pairing
$\theta$ was introduced in \cite{hochster1981dimension}; for more
detailed discussions related to this pairing, we refer the reader to
\cite{buchweitz2012index}, \cite{dao2013decent}, and
\cite{moore2011hochster}.
\end{rem}

\begin{rem}
Under our assumptions, by Theorem 4.11 of \cite{dyckerhoff2011compact}, the map
$$K_0[\MF(Q,f)] \to K_0[\MF(Q_{\m}, f)]$$
induced by inclusion is an isomorphism, so we may think of $\theta$ as a pairing on $K_0[\MF(Q,f)]$.
\end{rem}

Let $P = (P_1 \darrow{d_1}{d_0} P_0)$ be a matrix factorization of $f$ over $Q$. We observe that the image of $\phi^\C_f([P])$ under the
isomorphism $L_1(B_\epsilon, F_f) \cong KU^{-1}(F_f)$ coincides with $\alpha(\coker(d_1)_\m)|_{F_f}$, where $\alpha$ is as in Section 4 of
\cite{buchweitz2012index}. Thus, Proposition 4.1 and
Theorem 4.2 of \cite{buchweitz2012index} immediately imply:

\begin{prop}
\label{vanishing}
If $X \in \ker (\phi^\C_f)$, $\theta(X, -): K_0[\MF(Q, f)] \to \Z$
is the zero map.
\end{prop}

Set $K_0[\MF(Q, f)]_{\tors}$ to be the torsion subgroup of
$K_0[\MF(Q, f)]$. We conclude this section with the following explicit
description of $\ker(\phi^\C_f)$ when $n=2$:

\begin{prop}
\label{longresult}
If $f \in (x_1, x_2) \subseteq Q=\C[x_1, x_2]$, and the hypersurface
$Q/(f)$ has an isolated singularity at the origin in the sense of Definition~\ref{IHS}, $\ker(\phi^\C_f) = K_0[\MF(Q, f)]_{\tors}$.
\end{prop}

\begin{proof}
$K_0[\MF(Q, f)]_{\tors} \subseteq \ker(\phi^\C_f)$ is
obvious. Suppose $[P] \in \ker(\phi^\C_f)$. By
Proposition~\ref{vanishing}, the map $\theta([P], -): K_0[\MF(Q,f)] \to \Z$
is the zero map. Set $R= Q/(f)$. Since
$K_0[\MF(Q_{(x_1, x_2)}, f)] \cong G_0(R_{(x_1, x_2)})/[R_{(x_1, x_2)}],$ an application of
Proposition 3.3 in \cite{dao2013decent} finishes the proof.
\end{proof}


\subsection{Kn\"orrer periodicity and Bott periodicity}
We now use our
constructions $\phi^\R_f$ and $\phi^\C_f$ to exhibit
a compatibility between Kn\"orrer periodicity (Theorem~\ref{kpint}) and
Bott periodicity. Set

$$Q:=\R[x_1, \dots, x_n] \text{, }Q':=\R[y_1, \dots, y_m]$$
and let 
$$f \in (x_1, \dots, x_n)
\subseteq Q \text{, } f' \in (y_1, \dots, y_m) \subseteq Q'$$ 
be quasi-homogeneous polynomials.

\begin{rem}
We are assuming $f$ and $f'$ are quasi-homogeneous so that the version
of the
Sebastiani-Thom homotopy equivalence for real polynomials is available to us (see
Section~\ref{realmilnorfiber}). Analogous versions of every result in
this section hold over $\C$ when both $f \in \C[x_1, \dots, x_n]$ and
$f' \in \C[y_1, \dots, y_m]$ are either quasi-homogeneous or $\IHS$.
\end{rem}

We now construct the negative Milnor fibers of $f$ and $f'$. Choose real numbers $\epsilon'', \delta'',$ such that the map
$$B_{\epsilon''} \cap (f \oplus f')^{-1}((-\delta'',0)) \to (-\delta'',0)$$
given by $x \mapsto (f \oplus f')(x)$ is a locally trivial fibration. Similarly, choose $\epsilon, \delta$ and $\epsilon', \delta'$, as well
as $t'' \in (-\delta'', 0)$,
so that the analogous maps
$$B_{\epsilon} \cap f^{-1}((-\delta, 0)) \to (-\delta, 0)$$
$$B_{\epsilon'} \cap (f')^{-1}((-\delta', 0)) \to (-\delta', 0)$$
are locally trivial fibrations, and also so that
\begin{itemize}
\item[(a)] $\epsilon, \epsilon'$ are sufficiently small so that
$B_\epsilon \times B_{\epsilon'} \subseteq B_{\epsilon''}$.
\item[(b)] $|t''| < \min\{\delta, \delta'\}$.
\end{itemize}

Set $F_f^-$, $F_{f'}^-$, and $F_{f \oplus f'}^-$ to be the negative Milnor fibers of
$f$, $f'$, and $f \oplus f'$ over $t''$. Assume they are nonempty.

\begin{rem}
We could proceed using positive Milnor fibers as well, but we use
negative fibers to stay consistent with Section~\ref{ABSsection}.
\end{rem}

Recall from Remark~\ref{K0pairing} that we have a map
$$K_0[\MF(Q, f)] \otimes K_0[\MF(Q', f')] \to K_0[\MF(Q \otimes_\R Q',
f \oplus f')]$$
given by $[P] \otimes [P'] \mapsto [P \otimes_{\MF} P'].$ The following proposition is the key technical result in this section.

\begin{prop}
\label{compatible}
There exists a map
$$\ST_{L_1}: L_1(B_\epsilon, F_f^-) \otimes L_1(B_{\epsilon'}, F_{f'}^-) \to L_1(B_{\epsilon''}, F^-_{f
  \oplus f'})$$
such that, given matrix factorizations $P$ and $P'$ of $f$ and $f'$,
respectively,
$$\ST_{L_1}(\phi^\R_f([P]) \otimes \phi^\R_{f'}([P'])) = \phi^\R_{f \oplus
  f'}([P \otimes_\MF P']).$$
\end{prop}

\begin{proof}
Write
$$P=(P_1 \darrow{d_1}{d_0} P_0) \text{, } P'=(P_1' \darrow{d_1'}{d_0'} P_0')$$
and
$$\Phi^\R_f(P)=[V_1, V_0; d_1|_{F_f^-}] \text{, } \Phi^\R_{f'}(P')=[V_1', V_0'; d_1'|_{F_{f'}^-}].$$
We note that $$\phi^\R_{f \oplus
  f'}([P \otimes_\MF P'])=[(V_1 \otimes V_0') \oplus (V_0 \otimes V_1'), (V_0 \otimes V_0')
\oplus (V_1 \otimes V_1'); A],$$
where $A$ is the restriction of the matrix 
$$\begin{pmatrix} d_1 \otimes \id  & \id
  \otimes d_1'  \\ -\id \otimes d'_0 &
 d_0 \otimes
\id\end{pmatrix}$$
to $F^-_{f
  \oplus f'}$.

As in Section~\ref{TStopp}, choose a continuous injection $H:  CF_f^-
\to B_\epsilon$ such that
\begin{itemize}
\item $H(x, 1) = x \in F_f^- \subseteq B_\epsilon$,
\item $H(-, s): F_f^- \to B_\epsilon$ maps into the Milnor fiber $B_\epsilon
  \cap f^{-1}(st'')$ for $s \in (0, 1)$, and 
\item $H(x, 0) = 0$ for all $x \in F_f^-$
\end{itemize}

Choose $H': CF_{f'}^- \to B_{\epsilon'}$ similarly. The maps of pairs 
$$l: (CF_f^- , F_f^-)
\to (B_\epsilon, F_f^-) \text{, } l': (CF_{f'}^-, F_{f'}^-)
\to (B_{\epsilon'}, F_{f'}^-)$$ 
induced by $H$ and $H'$ yield isomorphisms on $L_1$ upon
pullback; this is immediate from the long exact sequence in
$KO$-theory and the naturality of the map $\chi$
from Section~\ref{$K$-theory} with respect to maps of pairs.

Recall from Section~\ref{$K$-theory} that we have a map
$$L_1(CF_{f}^-, F_f^-) \otimes L_1(CF_{f'}^-, F_{f'}^-) \to L_1(CF_{f}^- \times CF_{f'}^-, CF_{f}^-
\times F_{f'}^- \cup F_f^- \times CF_{f'}^-)$$
denoted by 
$$[V] \otimes [V'] \mapsto [V] \otimes_{L_1} [V'].$$

Define
$$\ST_{L_1}: L_1(B_\epsilon, F_f^-) \otimes L_1(B_{\epsilon'}, F_{f'}^-) \to L_1(B_{\epsilon''}, F^-_{f
  \oplus f'})$$
to be given by
$$[V] \otimes [V'] \mapsto (G^*)^{-1}(l^*([V]) \otimes_{L_1}
(l')^*([V'])),$$
where 
$$G: (CF_{f}^-\times CF_{f'}^-, CF_{f}^- \times F^-_{f'} \cup F^-_f \times
CF_{f'}^-) \to (B_{\epsilon''}, F^-_{f \oplus f'})$$
is as in Section~\ref{realmilnorfiber}. Recall that $G$ is an explicit
formulation of the Sebastiani-Thom
homotopy equivalence.

We now compute 
$l^*(\phi^\R_f(P)) \otimes_{L_1} (l')^*(\phi^\R_f(P'))$
explicitly. A splitting of the restriction of
$$\begin{pmatrix} -\id \otimes (H')^*(d_1') \\ H^*(d_1) \otimes \id \end{pmatrix}$$
to $ CF_{f}^-
\times F_{f'}^- \cup F_f^- \times CF_{f'}^-$ is given,
on the fiber over $(x,s,y, s')$, by 
$$\frac{1}{f(H(x,s)) + f'(H'(y,s'))} \begin{pmatrix} -\id \otimes (H')^*(d_0') & H^*(d_0)
  \otimes \id \end{pmatrix}$$
(notice that $f(H(x,s)) + f'(H'(y,s'))=(s+s')t'' \ne 0$ when $(x,s, y, s') \in CF_{f}^-
\times F_{f'}^- \cup F_f^- \times CF_{f'}^-$, since either $s$ or $s'$
is equal to 1). Thus, by the discussion at the end of Section~\ref{$K$-theory}, the product
$$l^*([V_1, V_0;
d|_{F_f^-}]) \otimes_{L_1}  (l')^*([V_1', V_0'; d'|_{F_{f'}^-}])$$ 
is equal to
$$[(H^*(V_1) \otimes (H')^*(V_0')) \oplus (H^*(V_0)
\otimes (H')^*(V_1')), (H^*(V_0) \otimes (H')^*(V_0')) \oplus (H^*(V_1)
\otimes (H')^*(V_1')); B],$$
where $B$ is given, on the fiber over $(x,s, y, s') \in CF_{f}^-
\times F_{f'}^- \cup F_f^- \times CF_{f'}^-$, by the matrix
$$\begin{pmatrix} H^*(d_1) \otimes \id  & \id  \otimes (H')^*(d_1')  \\ \frac{1}{f(H(x,s)) + f'(H'(y,s'))} (-\id \otimes (H')^*(d'_0)) &
\frac{1}{f(H(x,s)) + f'(H'(y,s'))}  (H^*(d_0) \otimes
\id)\end{pmatrix}.$$

We wish to show that, upon applying $(G^*)^{-1}$ to this class, one obtains
$$[(V_1 \otimes V_0') \oplus (V_0 \otimes V_1'), (V_0 \otimes V_0')
\oplus (V_1 \otimes V_1');C],$$
where $C$ is the restriction of the matrix
$$\begin{pmatrix} d_1 \otimes \id  & \id  \otimes d_1'  \\ \frac{1}{t''} (-\id \otimes d'_0) &
\frac{1}{t''}  (d_0 \otimes
\id)\end{pmatrix}$$
to $F^-_{f \oplus f'}$.
This will finish the proof, since the class
$$[(V_1 \otimes V_0') \oplus (V_0 \otimes V_1'), (V_0 \otimes V_0')
\oplus (V_1 \otimes V_1');C]$$ 
is clearly equal to
$$[(V_1 \otimes V_0') \oplus (V_0 \otimes V_1'), (V_0 \otimes V_0')
\oplus (V_1 \otimes V_1');A].$$

Observe that we have an object
$$[((H^*(V_1) \otimes (H')^*(V_0')) \oplus (H^*(V_0) \otimes (H')^*(V_1'))) \times I, ((H^*(V_0) \otimes (H')^*(V_0'))
\oplus (H^*(V_1) \otimes (H')^*(V_1'))) \times I; D]$$
in $\fC_1(CF_{f}^-\times CF_{f'}^- \times
I, (CF_{f}^- \times F^-_{f'} \cup F^-_f \times
CF_{f'}^-) \times I),$
where $D$ is given, on the fiber over $$(x,s,y, s', T) \in (CF_{f}^- \times F^-_{f'} \cup F^-_f \times
CF_{f'}^-) \times I,$$ by the matrix
$$\begin{pmatrix} H^*(d_1) \otimes \id  & \id  \otimes (H')^*(d_1') \\ \frac{1}{f(a(T)) + f'(b(T))} (-\id \otimes (H')^*(d'_0)) &
\frac{1}{f(a(T)) + f'(b(T))}  (H^*(d_0) \otimes
\id)\end{pmatrix}.$$

Here, $f$, $f'$, and the entries of $d_1, d_1', d_0, d_0'$ are
evaluated at the point
$$(a(T), b(T)):=(H(x, \frac{T(1-s'-s) +2s}{2}), H'(y, \frac{T(1-s'-s) + 2s'}{2})).$$

Notice that $f(a(T)) + f'(b(T)) \ne 0$ for all $$(x,s,y, s', T) \in (CF_{f}^- \times F^-_{f'} \cup F^-_f \times
CF_{f'}^-) \times I,$$ so this matrix is indeed an isomorphism on
every fiber over $(CF_{f}^- \times F^-_{f'} \cup F^-_f \times
CF_{f'}^-) \times I$. 

Restricting to $T=0$, one obtains
the object
$$((H^*(V_1) \otimes (H')^*(V_0')) \oplus (H^*(V_0)
\otimes (H')^*(V_1')), (H^*(V_0) \otimes (H')^*(V_0')) \oplus (H^*(V_1)
\otimes (H')^*(V_1')); B).$$
Restricting to $T=1$ and applying $(G^*)^{-1}$, one obtains
$$((V_1 \otimes V_0') \oplus (V_0 \otimes V_1'), (V_0 \otimes V_0')
\oplus (V_1 \otimes V_1');C).$$ 
Now apply Proposition 9.2 in \cite{atiyah1964clifford}.
\end{proof}

\begin{rem}
\label{induced}
It follows easily from the naturality of the map
$\chi$ from Section~\ref{$K$-theory} and Remark~\ref{10.4} that $\ST_{L_1}$ induces a map
$$\ST_{KO}: KO^0(B_\epsilon, F_f^-) \otimes KO^0(B_{\epsilon'}, F_{f'}^-) \to KO^0(B_{\epsilon''}, F^-_{f
  \oplus f'}).$$
\end{rem}

\begin{rem}
\label{reasonforname} We emphasize that the group homomorphism
$\ST_{L_1}$ in Proposition~\ref{compatible} is given by the
composition of the product in topological $K$-theory with the
inverse of the pullback along a 
specific formulation of the Sebastiani-Thom homotopy
equivalence. Hence, Proposition~\ref{compatible} yields a precise
sense in which the tensor product of matrix factorizations is related
to the Sebastiani-Thom homotopy equivalence.
\end{rem}

Let us now consider the case where $Q'=\R[u_1, \dots, u_8]$ and $f' =
-u_1^2 -\cdots - u_8^2$. By Theorem~\ref{rk} and
Remark~\ref{powerseries}, $[\MF(Q', f')] \cong [\MF(\R, 0)]$. It
follows that $K_0[\MF(\R[u_1, \dots, u_8],
-u_1^2- \cdots - u_8^2)]$ is isomorphic to $\Z$, generated by the class represented by the
matrix factorization $X$ constructed in the proof of Theorem~\ref{rk}.

Also, $F^-_{-u_1^2 - \cdots - u_8^2}$ is
homeomorphic to $S^7$, and so
$L_1(B_{\epsilon'}, F^-_{-u_1^2 - \cdots - u_8^2})$ is isomorphic to $\Z$. This group is generated by
$\phi^\R_{-u_1^2 - \cdots - u_8^2}(X);$ thus, $\phi^\R_{-u_1^2 - \cdots - u_8^2}([X])$ is a
Bott element in the group \newline $L_1(B_{\epsilon'}, F_{-u_1^2 - \cdots - u_8^2}) \cong
\widetilde{KO}^0(S^8)$; we shall denote by $\beta_\R$ the map
$$KO^0(B_\epsilon, F_f^-) \to KO^0(B_\epsilon, F_f^-) \otimes KO^0(B_{\epsilon'}, F_{-u_1^2 - \cdots - u_8^2})$$
given by $(\chi \otimes \chi) \circ (- \otimes \phi^\R_{-u_1^2 - \cdots - u_8^2}([X]))
\circ \chi^{-1}$. $\beta_\R$ is the Bott periodicity isomorphism.

Since real Kn\"orrer periodicity may be induced by tensoring with the matrix
factorization $X$, we will denote by
$K_\R$ the
map 
$$K_0[\MF(Q, f)] \to K_0[\MF(Q[u_1, \dots, u_8],
f -u_1^2- \cdots - u_8^2)]$$
given by $- \otimes_\MF [X]$.

The following result gives a precise sense in which Bott periodicity
and Kn\"orrer periodicity are compatible; it follows immediately from
Proposition~\ref{compatible}. We emphasize that a virtually identical proof
yields a similar result involving positive Milnor fibers.

\begin{thm}
\label{KPBPR}
Let $f \in Q=\R[x_1, \dots, x_n]$ be a quasi-homogeneous polynomial such
that $F_f^- \ne \emptyset$, and set $q=-u_1^2 - \cdots -u_8^2$. Then the diagram
$$
\xymatrix{
K_0[\MF(Q, f)] \ar[r]^-{\chi \circ \phi^\R_f} \ar[dd]^-{K_\R} & KO^0(B_\epsilon, F_f^-) \ar[d]^-{\beta_\R} \\
& KO^0(B_\epsilon, F_f^-) \otimes KO^0(B_{\epsilon'},
  F^-_{q}) \ar[d]^-{\ST_{KO}} \\
  K_0[\MF(Q[u_1, \dots, u_8],
f + q)] \ar[r]^-{\chi \circ \phi^\R_{f + q}} & KO^0(B_{\epsilon''}, F^-_{f
  + q}) 
}
$$
commutes.
\end{thm}

We state the analogous version of this result over the complex
numbers. Let $Y$ denote the matrix factorization $\C[u,v]
\darrow{u+iv}{u-iv} \C[u,v]$ of $u^2 + v^2$, and let 
$$K: K_0[\MF(Q, f)] \to K_0[\MF(Q[u,v],
f + u^2+v^2)]$$
be given by $- \otimes_\MF [Y]$. $K_0[\MF(\C[u,v], u^2 +
v^2)] \cong \Z$, and the group is generated by $[Y]$. Also, by
Theorem~\ref{bouquet}, $F_{u^2 +v^2}$ is homotopy equivalent to
$S^1$. Thus, $L_1(B_{\epsilon'}, F^-_{u^2 + v^2})$ is
isomorphic to $\Z$, generated by 
$\phi^\C_{u^2 + v^2}([Y])$. $\phi^\C_{u^2 + v^2}([Y])$ is a
Bott element in the group $L_1(B_{\epsilon'}, F_{u^2+v^2}) \cong
\widetilde{KU}^0(S^2)$; we shall denote by $\beta$ the map
$$KU^0(B_\epsilon, F_f) \to KU^0(B_\epsilon, F_f) \otimes
KU^0(B_{\epsilon'}, F_{u^2 + v^2})$$
given by $(\chi \otimes \chi) \circ (- \otimes \phi^\C_{u^2 + v^2}([Y]))
\circ \chi^{-1}$. $\beta$ is the complex Bott periodicity
isomorphism. Let $\ST_{KU}$ denote the pairing on relative
$KU$-theory induced by the complex version of the pairing
$\ST_{L_1}$. The following is a complex analogue of Theorem~\ref{KPBPR}:

\begin{thm}
\label{KPBP}
Let $f \in (x_1, \dots, x_n) \subseteq Q=\C[x_1, \dots, x_n]$ and
suppose either
\begin{itemize}
\item The hypersurface $\C[x_1,
\dots, x_n]_{(x_1, \dots, x_n)}/(f)$ is $\IHS$ (see
Definition~\ref{IHS}), or
\item $f$ is quasi-homogeneous.
\end{itemize}
Then the diagram
$$
\xymatrix{
K_0[\MF(Q, f)] \ar[r]^-{\chi \circ \phi^\C_f} \ar[dd]^{K} & KU^0(B_\epsilon, F_f) \ar[d]^{\beta} \\
& KU^0(B_\epsilon, F_f) \otimes KU^0(B_{\epsilon'},
  F_{u^2 + v^2}) \ar[d]^-{\ST_{KU}} \\
K_0[\MF(Q[u,v],
f + u^2+ v^2)] \ar[r]^-{\chi \circ \phi^\C_{f + u^2 + v^2}} & KU^0(B_{\epsilon''}, F_{f
  +u^2 + v^2}) \\
}
$$
commutes.
\end{thm}

\bibliographystyle{amsalpha}
\bibliography{Bibliography}

\providecommand{\bysame}{\leavevmode\hbox to3em{\hrulefill}\thinspace}
\providecommand{\MR}{\relax\ifhmode\unskip\space\fi MR }
\providecommand{\MRhref}[2]{%
  \href{http://www.ams.org/mathscinet-getitem?mr=#1}{#2}
}
\providecommand{\href}[2]{#2}
\begin{thebibliography}{MPSW11}

\bibitem[ABS64]{atiyah1964clifford}
Michael~F. Atiyah, Raoul Bott, and Arnold Shapiro, \emph{Clifford modules},
  Topology \textbf{3} (1964), 3--38.

\bibitem[AGZV12]{arnold2012singularities}
V.I. Arnold, Sabir~Medzhidovich Gusein-Zade, and Alexander~N Varchenko,
  \emph{Singularities of differentiable maps, volume 2: Monodromy and
  asymptotics of integrals}, vol.~2, Springer, 2012.

\bibitem[BEH87]{buchweitz1987cohen}
Ragnar-Olaf Buchweitz, David Eisenbud, and J{\"u}rgen Herzog,
  \emph{Cohen-{M}acaulay modules on quadrics}, Singularities, representation of
  algebras, and vector bundles, Springer, 1987, pp.~58--116.

\bibitem[Bro06]{brown2006topology}
Ronald Brown, \emph{Topology and groupoids}, AMC \textbf{10} (2006), 12.

\bibitem[Buc86]{buchweitz1986maximal}
Ragnar-Olaf Buchweitz, \emph{Maximal {C}ohen-{M}acaulay modules and
  {T}ate-cohomology over {G}orenstein rings}, preprint (1986).

\bibitem[BvS12]{buchweitz2012index}
Ragnar-Olaf Buchweitz and Duco van Straten, \emph{An index theorem for modules
  on a hypersurface singularity}, Mosc. Math. J \textbf{12} (2012), 237--259.

\bibitem[Dao13]{dao2013decent}
Hailong Dao, \emph{Decent intersection and {T}or-rigidity for modules over
  local hypersurfaces}, Transactions of the American Mathematical Society
  \textbf{365} (2013), no.~6, 2803--2821.

\bibitem[Dim92]{dimca1992singularities}
Alexandru Dimca, \emph{Singularities and topology of hypersurfaces}, Springer,
  1992.

\bibitem[DP92]{dimca1992real}
Alexandra Dimca and Laurentiu Paunescu, \emph{Real singularities and dihedral
  representations}, School of Mathematics and Statistics, University of Sydney,
  1992.

\bibitem[Dyc11]{dyckerhoff2011compact}
Tobias Dyckerhoff, \emph{Compact generators in categories of matrix
  factorizations}, Duke Mathematical Journal \textbf{159} (2011), no.~2,
  223--274.

\bibitem[Eis80]{eisenbud1980homological}
David Eisenbud, \emph{Homological algebra on a complete intersection, with an
  application to group representations}, Transactions of the American
  Mathematical Society \textbf{260} (1980), no.~1, 35--64.

\bibitem[Hoc81]{hochster1981dimension}
Melvin Hochster, \emph{The dimension of an intersection in an ambient
  hypersurface}, Algebraic geometry, Springer, 1981, pp.~93--106.

\bibitem[Hov07]{hovey2007model}
Mark Hovey, \emph{Model categories}, Mathematical Surveys and Monographs,
  no.~63, American Mathematical Soc., 2007.

\bibitem[KKP08]{katzarkov2008hodge}
Ludmil Katzarkov, Maxim Kontsevich, and Tony Pantev, \emph{Hodge theoretic
  aspects of mirror symmetry}, arXiv preprint arXiv:0806.0107 \textbf{1}
  (2008).

\bibitem[Kn{\"o}87]{knorrer1987cohen}
Horst Kn{\"o}rrer, \emph{Cohen-{M}acaulay modules on hypersurface singularities
  {I}}, Inventiones mathematicae \textbf{88} (1987), no.~1, 153--164.

\bibitem[KR08]{khovanov2008matrix}
Mikhail Khovanov and Lev Rozansky, \emph{Matrix factorizations and link
  homology}, Fundamenta Mathematicae \textbf{199} (2008), no.~1, 1--91.

\bibitem[Lam05]{lam2005introduction}
Tsit-Yuen Lam, \emph{Introduction to quadratic forms over fields}, American
  Mathematical Soc., 2005.

\bibitem[L{\^e}76]{le1976some}
D{\~u}ng~Tr{\'a}ng L{\^e}, \emph{Some remarks on relative monodromy}, Centre de
  math{\'e}matiques de l'{\'E}cole polytechnique, 1976.

\bibitem[LW12]{leuschke2012cohen}
Graham~J Leuschke and Roger Wiegand, \emph{Cohen-{M}acaulay representations},
  vol. 181, American Mathematical Society, 2012.

\bibitem[Mil68]{milnor1968singular}
John~W. Milnor, \emph{Singular points of complex hypersurfaces}, Annals of
  Mathematics Studies, no.~61, Princeton University Press, 1968.

\bibitem[MPSW11]{moore2011hochster}
W~Frank Moore, Greg Piepmeyer, Sandra Spiroff, and Mark~E Walker,
  \emph{Hochster's theta invariant and the {H}odge--{R}iemann bilinear
  relations}, Advances in Mathematics \textbf{226} (2011), no.~2, 1692--1714.

\bibitem[Oka73]{oka1973homotopy}
Mutsuo Oka, \emph{On the homotopy types of hypersurfaces defined by weighted
  homogeneous polynomials}, Topology \textbf{12} (1973), no.~1, 19--32.

\bibitem[Orl03]{orlov2003triangulated}
Dmitri Orlov, \emph{Triangulated categories of singularities and {D}-branes in
  {L}andau-{G}inzburg models}, arXiv preprint math/0302304 (2003).

\bibitem[ST71]{sebastiani1971resultat}
M.~Sebastiani and Ren{\'e} Thom, \emph{Un r{\'e}sultat sur la monodromie},
  Inventiones mathematicae \textbf{13} (1971), no.~1, 90--96.

\bibitem[To{\"e}07]{toen2007homotopy}
Bertrand To{\"e}n, \emph{The homotopy theory of dg-categories and derived
  {M}orita theory}, Inventiones mathematicae \textbf{167} (2007), no.~3,
  615--667.

\bibitem[To{\"e}11]{bertrand2011lectures}
\bysame, \emph{Lectures on dg-categories}, Topics in algebraic and topological
  K-theory, Springer, 2011, pp.~243--302.

\bibitem[Woo66]{wood1966banach}
R.~Wood, \emph{Banach algebras and {B}ott periodicity}, Topology \textbf{4}
  (1966), no.~4, 371--389.

\bibitem[Yos90]{yoshino1990maximal}
Yuji Yoshino, \emph{Maximal {C}ohen-{M}acaulay modules over {C}ohen-{M}acaulay
  rings}, vol. 146, Cambridge University Press, 1990.

\bibitem[Yos98]{yoshino1998tensor}
\bysame, \emph{Tensor products of matrix factorizations}, Nagoya Mathematical
  Journal \textbf{152} (1998), 39--56.

\end{thebibliography}

\end{document}